\documentclass{amsart}
\usepackage{a4}
\usepackage{hyperref}

\usepackage{comment}
\usepackage{fullpage}
\usepackage{cellspace}
\usepackage{xcolor}

\usepackage{amsmath}
\usepackage{amssymb}
\usepackage[backrefs,lite]{amsrefs}

\usepackage{longtable}
\usepackage{booktabs}

\usepackage{enumerate}
\usepackage{enumitem}

\usepackage[all]{xy}
\usepackage{tikz}		
\usepackage{tikz-3dplot}
\usepackage{tikz-cd}
\usepackage{graphicx}

\usetikzlibrary{patterns,matrix,arrows,shapes}

\theoremstyle{plain}

\newtheorem{theorem}{Theorem}[section]
\newtheorem{lemma}[theorem]{Lemma}
\newtheorem{proposition}[theorem]{Proposition}
\newtheorem{corollary}[theorem]{Corollary}

\theoremstyle{definition}

\newtheorem{notation}[theorem]{Notation}

\newtheorem{construction}[theorem]{Construction}

\newtheorem{remark}[theorem]{Remark}
\newtheorem{rmk}[theorem]{Remark}
\theoremstyle{remark}

\usepackage{colortbl}
\newcommand{\evnrow}{\rowcolor[gray]{0.95}}

\numberwithin{equation}{section}

\newcommand\KK{{\mathbb K}}
\newcommand\TT{{\mathbb T}}
\newcommand\ZZ{{\mathbb Z}}

\newcommand\QQ{{\mathbb Q}}
\newcommand\PP{{\mathbb P}}

\newcommand\OOO{{\mathcal O}}

\newcommand\rlv{{\rm rlv}}

\newcommand\Cl{\operatorname{Cl}}

\newcommand\Spec{{\rm Spec}}

\newcommand\Chi{{\mathbb X}}

\newcommand\bangle[1]{\langle #1 \rangle}

\newcommand\Gr{{\rm Gr}}
\newcommand\w{{\rm w}}
\newcommand\aGr{{\rm aGr}}

\newcommand\wGr{{\rm wGr}}
\begin{document}
%Titel
\title[Smooth Fano intrinsic Grassmannians of type $(2,n)$ with $\rho(X) = 2$]%
{Smooth Fano intrinsic Grassmannians of type $(2,n)$ with Picard number two
}
%Fano intrinsic GR(2,n)-manifolds of Picard rank two
%Subject-Classification
%\subjclass[2010]{14J45, 14L30, 14M25}

\author[M.~I.~Qureshi]{Muhammad Imran Qureshi} 
\address{Deptartment of Mathematics \& Statistics, King Fahd University of Petroleum and Minerals, Saudi Arabia}
\email{imran.qureshi@kfupm.edu.sa}

\author[M.~Wrobel]{Milena Wrobel} 
\address{Institut f\"ur Mathematik, Carl von Ossietzky Universit\"at Oldenburg, 26111 Oldenburg, Germany}
\email{milena.wrobel@uni-oldenburg.de}

\begin{abstract}
We introduce the notion of intrinsic Grassmannians which generalizes the well known weighted Grassmannians. An intrinsic Grassmannian is a normal projective variety whose Cox ring is defined by the Plücker ideal $I_{d,n}$ of the Grassmannian $\mathrm{Gr}(d,n)$. 
We give a complete classification of all
smooth Fano intrinsic Grassmannians of type $(2,n)$ with Picard number two and 
prove an explicit formula to compute the total number of such varieties for an arbitrary $n$. We study their geometry and show that they satisfy Fujita's freeness conjecture. 
\end{abstract}

\maketitle
\vspace{-10mm}

\section{Introduction}\label{sec:intro}
This paper contributes to the classification of Fano varieties that means
algebraic varieties with ample anticanonical divisor class.
More precisely, we consider  
\emph{intrinsic Grassmannians of type $(d,n)$}, i.e.
normal projective varieties $X$ with finitely generated divisor class group $\mathrm{Cl}(X)$ and finitely generated Cox ring
$$
\mathcal{R}(X) := \bigoplus_{\Cl(X)} \Gamma(X, \mathcal{O}_X(D))
$$
admitting homogeneous generators such that the associated ideal of relations
is the Plücker ideal $I_{d,n}$ of~$\mathrm{Gr}(d,n)$.

The best studied examples are the so-called \emph{weighted Grassmannians} that form the case of intrinsic Grassmannians of Picard number one and
have been introduced by Corti and Reid~\cite{wg} in order to construct families of polarized varieties as quasilinear sections of weighted Grassmannians.
Note that the only smooth weighted Grassmannians are the standard Grassmannian varieties.

We go one step beyond and consider the case of intrinsic Grassmannians of Picard number two.
Our first main result gives a complete classification of all
smooth full intrinsic Grassmannians of type $(2,n)$,
where \emph{full} means, that all generators occur in the Plücker relations generating $I_{2,n}$. 
Note that any intrinsic Grassmannian $X$ is
fully determined by its Cox ring and its ample cone.
More precisely, we can regain $X$
as a certain geometric invariant theory quotient; see Construction \ref{constr:intrinsGrassmannian}.

\begin{theorem}\label{thm:classification}
Let $X$ be a smooth full intrinsic Grassmannian of type~$(2,n)$ with Picard number two. Then
the $\mathrm{Cl}(X) = \ZZ^2$-grading of 
$$\mathcal{R}(X)
= {\KK[T_{ij}; \ 1 \leq i < j\leq n ]/I_{2,n}}$$ 
and the semiample cone are of the type below;
we write $w_{ij} := \mathrm{deg}(T_{ij})$
for the $\mathrm{Cl}(X)$-degrees.

\vspace{1mm}
\noindent
Fix an integer $4 \leq k \leq n$ and a sequence of integers
$ 0 =  \alpha_k \leq \alpha_{k+1} \leq \ldots \leq \alpha_n$.
We have
\begin{enumerate}
    \item 
    $w_{ij} = (1,0)$ whenever $j < k$ holds,
    \item
    $w_{ij} =  (\alpha_j, 1)$, whenever $i < k \leq j$ holds and
    \item 
    $w_{ij} = (\alpha_i + \alpha_j -1, 2)$ whenever $k \leq i$.
\end{enumerate}
The semiample cone looks as follows:

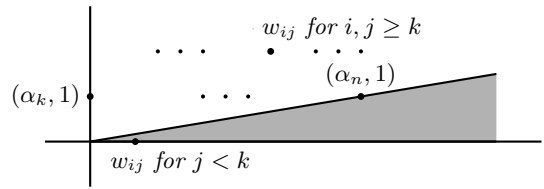
\begin{center}
\begin{tikzpicture}[scale=0.6]
%semiample cone
\path[fill=gray!60!] (0,0)--(9,1.5)--(9,0);
%rays of the semiample cone \tau_X
\draw[-,thick] (0,0)--(9,1.5);
\draw[-,thick] (0,0)--(9,0);
%Cl(X)-degrees of type (i)
\path[fill, color=black] (1,0) circle (0.5ex);
\path[fill, color=black] (2,0) circle (0.0ex)
node[below]{\small $w_{ij}$ for $j < k$};
%Cl(X)-degrees of type (ii)
\path[fill, color=black] (6,1) circle (0.5ex) 
node[above]{\small $(\alpha_n,1)$};
\path[fill, color=black] (0,1) circle (0.5ex)
node[left]{\small $(\alpha_k,1)$};
\path[fill, color=black] (2.5,1) circle (0.3ex);
\path[fill, color=black] (3,1) circle (0.3ex);
\path[fill, color=black] (3.5,1) circle (0.3ex);
%Cl(X)-degrees of type (iii) for i,j \geq k
\path[fill, color=black] (4,2) circle (0.5ex);
\path[fill, color=black] (3.6,2.5) circle (0.0ex)
node[right]{\small $w_{ij}$ for $i,j\geq k$};
\path[fill, color=black] (1.5,2) circle (0.3ex);
\path[fill, color=black] (2,2) circle (0.3ex);
\path[fill, color=black] (2.5,2) circle (0.3ex);
\path[fill, color=black] (5,2) circle (0.3ex);
\path[fill, color=black] (5.5,2) circle (0.3ex);
\path[fill, color=black] (6,2) circle (0.3ex);
%axes
\draw[-,thick] (-1,0)--(10,0);
\draw[-,thick] (0,-1)--(0,3);
\end{tikzpicture}   
\end{center}
Moreover, each of the configurations above gives a smooth full intrinsic Grassmannian of type $(2,n)$ with Picard number two.
\end{theorem}

More generally, in Theorem~\ref{thm:classNonFull} we carry out the 
more comprehensive classification in
the non-full case. 
We give a description of their geometry in Section~\ref{sec:geometry}
and show that every smooth intrinsic Grassmannian of type $(2,n)$ with Picard number two fulfills Fujita's freeness conjecture in Corollary~\ref{cor:Fujita}.

Note that intrinsic Grassmannians of type $(2,4)$
are \emph{intrinsic quadrics,}
i.e.\ normal, projective varieties with finitely generated divisor class group and finitely generated Cox ring admitting homogeneous generators such that the 
ideal of relations is generated by a single quadric.
In the smooth case their classification has been carried out in \cite{FaHa}
for small Picard numbers.
In particular, we will use parts of their results for the proof of the above statement.

Using our explicit description of the anticanonical class of a full intrinsic Grassmannian of type $(2,n)$, see Proposition \ref{prop:AnticanClass},
we are able to characterize the Fano and truly almost Fano varieties among the 
smooth full intrinsic Grassmannians in the above theorem; see Corollary~\ref{cor:FanoNonFull} for the analogues statement in the non-full case.
Recall that a normal projective variety $X$ is called
\emph{almost Fano} if it has a numerically effective anticanonical divisor and we call $X$ \emph{truly almost Fano} if it is almost Fano but not Fano.

\begin{corollary}\label{cor:FanoClassification}
In the notation of Theorem~\ref{thm:classification}
the (truly almost) Fano varieties among the smooth
full intrinsic Grassmannians of type $(2,n)$ with Picard number two are characterized as follows:

\begin{longtable}{c|c}
Fano & truly almost Fano 
\\
\hline

\begin{minipage}[t]{0.45\textwidth}
\vspace{-3mm}
$$\sum_{i=k+1}^{n-1} \alpha_i - \frac{1}{2} n + k - 1 > (n-k)\alpha_n$$
\end{minipage}
&
\begin{minipage}[t]{0.45\textwidth}
\vspace{-3mm}
$$\sum_{i=k+1}^{n-1} \alpha_i - \frac{1}{2} n + k - 1 = (n-k)\alpha_n$$
\end{minipage}
\end{longtable}
\end{corollary}

The inequality of the Fano criterion forces $\alpha_n < k - \frac{1}{2}n -1$. 
In particular, the full flag variety 
$\mathrm{V}(T_{0}S_{0} + T_1S_1 + T_2S_2) \subseteq \PP^2 \times \PP^2$
is the only smooth Fano full intrinsic Grassmannian of type $(2,4)$ with Picard number two; see also \cite{FaHa}*{Thm. 1.3}.

For further illustration we exemplarily list all smooth Fano full intrinsic Grassmannians of type $(2,n)$
for $5 \leq n \leq 8$ and in addition computed their first plurigenera $h^0(-\mathcal K_X)$ using the computer algebra system Macaulay2 \cite{M2}.

\begin{corollary}
Every smooth Fano full intrinsic Grassmannian of type $(2,n)$ with Picard number two
and $5 \leq n \leq 8$ is isomorphic to one of the varieties $X$ in the list below,
specified by their Cox ring 
$$\mathcal{R}(X) = \KK[T_{ij}; \ 1 \leq i < j\leq n ]/I_{2,n},$$
 and the anticanonical class $-\mathcal{K}_X \in \mathrm{Cl}(X).$ 
 In all cases $\mathrm{Cl}(X) \cong \ZZ^2$ 
and the grading is fixed by the matrix
$[w_{12}, w_{13}, \ldots, w_{(n-1)n}]$
of lexicographically ordered generator degrees 
$\mathrm{deg}(T_{ij})$.

\vspace{3mm}
{ \setlength{\arraycolsep}{1.8pt}
\renewcommand*{\arraystretch}{1}
\begin{longtable}{|Sc||Sc|Sc|Sc|Sc|}
\hline\evnrow{\small
No.}
&
{\small
$n$}
&
{\small
$[w_{12}, w_{13}, \ldots,w_{1n}, w_{23}, \ldots, w_{(n-1)n}]$}
&
{\small
$- \mathcal{K}_X$ }
&
{\small
$h^0(- \mathcal{K}_X)$ }
\\
\hline
 1&
$5$
&
{\tiny
$\left[\begin{array}{cccccccccc}
1 & 1 & 1 & 0 & 1 & 1 & 0 & 1 & 0 & 0
\\
0 & 0 & 0 & 1 & 0 & 0 & 1 & 0 & 1 & 1
\end{array}\right]$
}
&
{\tiny
$
\left[
\begin{array}{c}
3\\
2
\end{array}
\right]
$
}& {\small$280$}
\\
\hline
%%%%%%%%%%%%%%%%%%%%%%%%%%%%%%%%%%%%%%%%%%%%%%%%%%
%%%%%%%%%%%%%%%%%%%%%%%%%%%%%%%%%%%%%%%%%%%%%%%%%%
%%%%%%%%%%%%%%%%%%%%%%%%%%%%%%%%%%%%%%%%%%%%%%%%%%
\evnrow 2&
$5$
&
{\tiny
$\left[\begin{array}{cccccccccc}
1 & 1 & 0 & 0 & 1 & 0 & 0 & 0 & 0 & -1
\\
0 & 0 & 1 & 1 & 0 & 1 & 1 & 1 & 1 & 2
\end{array}\right]$
}
&
{\tiny
$
\left[
\begin{array}{c}
1\\
4
\end{array}
\right]
$
}&\small{$266$}
\\
\hline
\hline
%%%%%%%%%%%%%%%%%%%%%%%%%%%%%%%%%%%%%%%%%%%%%%%%%%
%%%%%%%%%%%%%%%%%%%%%%%%%%%%%%%%%%%%%%%%%%%%%%%%%%
%%%%%%%%%%%%%%%%%%%%%%%%%%%%%%%%%%%%%%%%%%%%%%%%%%
3&
$6$
&
{\tiny
$\left[\begin{array}{ccccccccccccccc}
1 & 1 & 1 & 1 & 0 & 1 & 1 & 1 & 0 & 1 & 1 & 0 & 1 & 0 & 0
\\
0 & 0 & 0 & 0 & 1 & 0 & 0 & 0 & 1 & 0 & 0 & 1 & 0 & 1 & 1
\end{array}\right]$
}
&
{\tiny
$
\left[
\begin{array}{c}
4\\
2
\end{array}
\right]
$
}&\small{$3750$}
\\
\hline
%%%%%%%%%%%%%%%%%%%%%%%%%%%%%%%%%%%%%%%%%%%%%%%%%%
%%%%%%%%%%%%%%%%%%%%%%%%%%%%%%%%%%%%%%%%%%%%%%%%%%
%%%%%%%%%%%%%%%%%%%%%%%%%%%%%%%%%%%%%%%%%%%%%%%%%%
\evnrow4&
$6$
&
{\tiny
$\left[\begin{array}{ccccccccccccccc}
1 & 1 & 1 & 0 & 0 & 1 & 1 & 0 & 0 & 1 & 0 & 0 & 0 & 0 & -1
\\
0 & 0 & 0 & 1 & 1 & 0 & 0 & 1 & 1 & 0 & 1 & 1 & 1 & 1 & 2
\end{array}\right]$
}
&
{\tiny
$
\left[
\begin{array}{c}
2\\
4
\end{array}
\right]
$
}&\small{$2745$}
\\
\hline
\hline
%%%%%%%%%%%%%%%%%%%%%%%%%%%%%%%%%%%%%%%%%%%%%%%%%%
%%%%%%%%%%%%%%%%%%%%%%%%%%%%%%%%%%%%%%%%%%%%%%%%%%
%%%%%%%%%%%%%%%%%%%%%%%%%%%%%%%%%%%%%%%%%%%%%%%%%%
%%%%%%%%%%%%%%%%%% aGR(2,7) %%%%%%%%%%%%%%%%%%%%%%
%%%%%%%%%%%%%%%%%%%%%%%%%%%%%%%%%%%%%%%%%%%%%%%%%%
%%%%%%%%%%%%%%%%%%%%%%%%%%%%%%%%%%%%%%%%%%%%%%%%%%
%%%%%%%%%%%%%%%%%%%%%%%%%%%%%%%%%%%%%%%%%%%%%%%%%%
5&
$7$
&
{\tiny
$\left[\begin{array}{ccccccccccccccccccccc}
1 & 1 & 1 & 1 & 1 & 0 & 1 & 1 & 1 & 1 & 0 & 1 & 1 & 1 & 0 & 1 & 1 & 0 & 1 & 0 & 0
\\
0 & 0 & 0 & 0 & 0 & 1 & 0 & 0 & 0 & 0 & 1 & 0 & 0 & 0 & 1 & 0 & 0 & 1 & 0 & 1 & 1
\end{array}\right]$
}
&
{\tiny
$
\left[
\begin{array}{c}
5\\
2
\end{array}
\right]
$
}&\small{$37422$}
\\
\hline
%%%%%%%%%%%%%%%%%%%%%%%%%%%%%%%%%%%%%%%%%%%%%%%%%%
%%%%%%%%%%%%%%%%%%%%%%%%%%%%%%%%%%%%%%%%%%%%%%%%%%
%%%%%%%%%%%%%%%%%%%%%%%%%%%%%%%%%%%%%%%%%%%%%%%%%%
\evnrow6&
$7$
&
{\tiny
$\left[\begin{array}{ccccccccccccccccccccc}
1 & 1 & 1 & 1 & 0 & 0 & 1 & 1 & 1 & 0 & 0 & 1 & 1 & 0 & 0 & 1 & 0 & 0 & 0 & 0 & -1
\\
0 & 0 & 0 & 0 & 1 & 1 & 0 & 0 & 0 & 1 & 1 & 0 & 0 & 1 & 1 & 0 & 1 & 1 & 1 & 1 & 2
\end{array}\right]$
}
&
{\tiny
$
\left[
\begin{array}{c}
3\\
4
\end{array}
\right]
$
}&\small{$31251$}
\\
\hline
%%%%%%%%%%%%%%%%%%%%%%%%%%%%%%%%%%%%%%%%%%%%%%%%%%
%%%%%%%%%%%%%%%%%%%%%%%%%%%%%%%%%%%%%%%%%%%%%%%%%%
%%%%%%%%%%%%%%%%%%%%%%%%%%%%%%%%%%%%%%%%%%%%%%%%%%
7&
$7$
&
{\tiny
$\left[\begin{array}{ccccccccccccccccccccc}
1 & 1 & 1 & 1 & 0 & 1 & 1 & 1 & 1 & 0 & 1 & 1 & 1 & 0 & 1 & 1 & 0 & 1 & 0 & 1 & 0
\\
0 & 0 & 0 & 0 & 1 & 1 & 0 & 0 & 0 & 1 & 1 & 0 & 0 & 1 & 1 & 0 & 1 & 1 & 1 & 1 & 2
\end{array}\right]$
}
&
{\tiny
$
\left[
\begin{array}{c}
5\\
4
\end{array}
\right]
$
}&\small{$48206$}
\\
\hline
%%%%%%%%%%%%%%%%%%%%%%%%%%%%%%%%%%%%%%%%%%%%%%%%%%
%%%%%%%%%%%%%%%%%%%%%%%%%%%%%%%%%%%%%%%%%%%%%%%%%%
%%%%%%%%%%%%%%%%%%%%%%%%%%%%%%%%%%%%%%%%%%%%%%%%%%
\evnrow8&
$7$
&
{\tiny
$\left[\begin{array}{ccccccccccccccccccccc}
1 & 1 & 1 & 0 & 0 & 0 & 1 & 1 & 0 & 0 & 0 & 1 & 0 & 0 & 0 & 0 & 0 & 0 & -1 & -1 & -1
\\
0 & 0 & 0 & 1 & 1 & 1 & 0 & 0 & 1 & 1 & 1 & 0 & 1 & 1 & 1 & 1 & 1 & 1 & 2 & 2 & 2
\end{array}\right]$
}
&
{\tiny
$
\left[
\begin{array}{c}
1\\
6
\end{array}
\right]
$
}&\small{$30030$}
\\
\hline
\hline
%%%%%%%%%%%%%%%%%%%%%%%%%%%%%%%%%%%%%%%%%%%%%%%%%%
%%%%%%%%%%%%%%%%%%%%%%%%%%%%%%%%%%%%%%%%%%%%%%%%%%
%%%%%%%%%%%%%%%%%%%%%%%%%%%%%%%%%%%%%%%%%%%%%%%%%%
9&
$8$
&
{\tiny
$\left[\begin{array}{cccccccccccccccccccccccccccc}
1 & 1 & 1 & 1 & 1 & 1 & 0 & 1 & 1 & 1 & 1 & 1 & 0 & 1 & 1 & 1 & 1 & 0 & 1 & 1 & 1 & 0 & 1 & 1 & 0 & 1 & 0 & 0
\\
0 & 0 & 0 & 0 & 0 & 0 & 1 & 0 & 0 & 0 & 0 & 0 & 1 & 0 & 0 & 0 & 0 & 1 & 0 & 0 & 0 & 1 & 0 & 0 & 1 & 0 & 1 & 1
\end{array}\right]$
}
&
{\tiny
$
\left[
\begin{array}{c}
6\\
2
\end{array}
\right]
$
}&\small{$462462$}
\\
\hline
%%%%%%%%%%%%%%%%%%%%%%%%%%%%%%%%%%%%%%%%%%%%%%%%%%
%%%%%%%%%%%%%%%%%%%%%%%%%%%%%%%%%%%%%%%%%%%%%%%%%%
%%%%%%%%%%%%%%%%%%%%%%%%%%%%%%%%%%%%%%%%%%%%%%%%%%
\evnrow10&
$8$
&
{\tiny
$\left[\begin{array}{cccccccccccccccccccccccccccc}
1 & 1 & 1 & 1 & 1 & 0 & 0 & 1 & 1 & 1 & 1 & 0 & 0 & 1 & 1 & 1 & 0 & 0 & 1 & 1 & 0 & 0 & 1 & 0 & 0 & 0 & 0 & -1
\\
0 & 0 & 0 & 0 & 0 & 1 & 1 & 0 & 0 & 0 & 0 & 1 & 1 & 0 & 0 & 0 & 1 & 1 & 0 & 0 & 1 & 1 & 0 & 1 & 1 & 1 & 1 & 2
\end{array}\right]$
}
&
{\tiny
$
\left[
\begin{array}{c}
4\\
4
\end{array}
\right]
$
}&\small{$376376$}
\\
\hline
%%%%%%%%%%%%%%%%%%%%%%%%%%%%%%%%%%%%%%%%%%%%%%%%%%
%%%%%%%%%%%%%%%%%%%%%%%%%%%%%%%%%%%%%%%%%%%%%%%%%%
%%%%%%%%%%%%%%%%%%%%%%%%%%%%%%%%%%%%%%%%%%%%%%%%%%
11&
$8$
&
{\tiny
$\left[\begin{array}{cccccccccccccccccccccccccccc}
1 & 1 & 1 & 1 & 1 & 0 & 1 & 1 & 1 & 1 & 1 & 0 & 1 & 1 & 1 & 1 & 0 & 1 & 1 & 1 & 0 & 1 & 1 & 0 & 1 & 0 & 1 & 0
\\
0 & 0 & 0 & 0 & 0 & 1 & 1 & 0 & 0 & 0 & 0 & 1 & 1 & 0 & 0 & 0 & 1 & 1 & 0 & 0 & 1 & 1 & 0 & 1 & 1 & 1 & 1 & 2
\end{array}\right]$
}
&
{\tiny
$
\left[
\begin{array}{c}
6\\
4
\end{array}
\right]
$
}&\small{$640333$}
\\
\hline
%%%%%%%%%%%%%%%%%%%%%%%%%%%%%%%%%%%%%%%%%%%%%%%%%%
%%%%%%%%%%%%%%%%%%%%%%%%%%%%%%%%%%%%%%%%%%%%%%%%%%
%%%%%%%%%%%%%%%%%%%%%%%%%%%%%%%%%%%%%%%%%%%%%%%%%%
\evnrow12&
$8$
&
{\tiny
$\left[\begin{array}{cccccccccccccccccccccccccccc}
1 & 1 & 1 & 1 & 0 & 0 & 0 & 1 & 1 & 1 & 0 & 0 & 0 & 1 & 1 & 0 & 0 & 0 & 1 & 0 & 0 & 0 & 0 & 0 & 0 & -1 & -1 & -1
\\
0 & 0 & 0 & 0 & 1 & 1 & 1 & 0 & 0 & 0 & 1 & 1 & 1 & 0 & 0 & 1 & 1 & 1 & 0 & 1 & 1 & 1 & 1 & 1 & 1 & 2 & 2 & 2
\end{array}\right]$
}
&
{\tiny
$
\left[
\begin{array}{c}
2\\
6
\end{array}
\right]
$
}&\small{$348985$}
\\
\hline
%%%%%%%%%%%%%%%%%%%%%%%%%%%%%%%%%%%%%%%%%%%%%%%%%%
%%%%%%%%%%%%%%%%%%%%%%%%%%%%%%%%%%%%%%%%%%%%%%%%%%
%%%%%%%%%%%%%%%%%%%%%%%%%%%%%%%%%%%%%%%%%%%%%%%%%%

\end{longtable}}

\end{corollary}

As a direct consequence of the Fano criterion given in Corollary~\ref{cor:FanoClassification}, we obtain the following explicit formula for the number of smooth Fano full intrinsic Grassmannians of type $(2,n)$ 
for arbitrary $n$.

\begin{corollary}\label{cor:NumberOfFanos}
The number of pairwise non-isomorphic smooth Fano full intrinsic Grassmannians of type $(2,n)$ that have Picard number two is given via the following formula:
$$
\sum_{i = \lfloor\frac{n}{2}\rfloor+ 2}^n 
\quad
\sum_{j= 0}^{i-\lfloor\frac{n}{2}\rfloor -2} 
\quad
\sum_{k = 0}^{i - \lfloor\frac{n}{2}\rfloor-2 - j}
a(j(n-i-1)-k, n-i-1,j),
$$
where
$a(x,y,z)$
is the number of all integer sequences
$
0 \leq a_1 \leq \ldots \leq a_y \leq z
$
with ${a_1 + \ldots + a_y = x}$. 
\end{corollary}

\subsection*{Context and further outlook}
The construction of Fano varieties from well understood classes of algebraic varieties plays a pivotal role in their classification. 
In particular, the Grassmannians and  homogeneous varieties feature in various veins in this context:
The construction of families of smooth  Fano 4-folds as sections of  homogeneous vector bundles over Grassmannians \cite{Kuchle-95},  and  Mukai's linear section theorem \cite{mukai4} showing every prime Fano 3-fold of genus \(7\le g \le 10\)  is a linear section of a homogeneous variety,  are two such instances.
Moreover, in \cite{wg} Corti and Reid  introduced the notion of weighted  Grassmannians $\wGr(d,n)$, to apply the Mukai's linear section theorem in the weighted case, i.e.\ to hunt  interesting classes of algebraic varieties as  complete intersections of $\wGr(d,n)$. They showed that   69 out of 70   families of terminal Fano 3-folds in codimension 3 can be constructed  as complete intersections of  some $\w\Gr(2,5)$ or projective cone over it. This idea has been subsequently used in \cites{QS,QS-AHEP,QJSC,QJGP,QS2,BKZ}  for the construction of Calabi--Yau and canonical 3-folds as quasilinear sections of weighted homogeneous varieties of various types.
We believe these ideas can be applied in the setting of Picard rank greater or equal to two to construct new classes of Fano manifolds as hypersurfaces and complete intersections of intrinsic Grassmannians. The work of Hausen, Laface and Mauz \cite{HLM} can serve as a model in the hypersurface case.

On the other hand, over the last few years the approach of classifying varieties with a prescribed Cox ring has 
been effectively employed by various authors:
Here, we would like to mention the work 
on smooth and mildly singular Fano varieties with torus action of complexity one \cites{ FaHaNi, BeHaHuNi}
and their generalization to the so called arrangement varieties \cites{HaHiWr,HiWr1, HiWr2}
comprising the intrinsic quadrics treated in \cites{FaHa,Hi}.
Working in this spirit, the next logical steps would be the generalization to  higher Picard numbers and
the treatment of intrinsic Grassmannians of type $(d,n)$ with $d \geq 3$ or entering the field of mildly singular Fano intrinsic Grassmannians.

\subsection*{Acknowledgement}
 This project was incepted when MIQ was at the Universitat  T\"ubingen on a fellowship of the Alexander von Humboldt foundation. We would like to thank J\"urgen Hausen, Christoff Hische and Miles Reid for helpful discussions. MIQ was supported by a faculty startup research grant No.\ SR191006 of the Deanship of Scientific Research at the King Fahd University of
Petroleum and Minerals during this project.

\tableofcontents

\section{Generalities on Intrinsic Grassmannians}
Throughout the whole article $\KK$ is an algebraically closed field of characteristic zero.
Let $X$ be a Mori dream space, i.e.\ a normal, projective
variety with finitely generated divisor class group and
finitely generated Cox ring, see \cite{ArDeHaLa}:
$$
\mathcal{R}(X) := \bigoplus_{\mathrm{Cl}(X)} \Gamma(X, \mathcal{O}_X(D)).
$$
Then we can regain $X$ from $\mathcal{R}(X)$ as follows:
The $\mathrm{Cl}(X)$-grading on $\mathcal{R}(X)$ gives rise to an action of the quasitorus 
$H:= \mathrm{Spec}(\KK[\mathrm{Cl}(X)])$
on the \emph{total coordinate space}
$\overline{X} := \mathrm{Spec}(\mathcal{R}(X))$.
Let $u \in \mathrm{Cl}(X)$ be an ample class
and consider the associated $H$-invariant open subset of \emph{semistable points}
$$
\overline{X}^\mathrm{ss}(u) 
:=
\left\{x \in \overline{X}; \ f(x) \neq 0 \text{ for some } 
f \in \mathcal{R}(X)_{nu}, \ n > 0\right\} \subseteq \overline{X}
$$
Then $X$ equals the geometric invariant theory quotient
$\overline{X}^\mathrm{ss}(u) / \!\! /H$.

In this article we will take
the opposite viewpoint:
We are interested in finding all Mori dream spaces having the homogeneous coordinate ring of the Plücker embedded Grassmannian $\mathrm{Gr}(d,n)$
as prescribed Cox ring, where we endow the ring with a suitable grading. 
More precisely we have the following construction:

\begin{construction}[Intrinsic Grassmannians of type $(d,n)$]
\label{constr:intrinsGrassmannianDN}
Fix integers $n > d \geq 2$ and $m \geq 0$ and consider the polynomial ring
$$
\KK[T_{i_1, \ldots, i_d}, S_l; \ 1 \leq i_1 < \ldots < i_d \leq n, 1 \leq l \leq m]
$$
We will denote this ring with $\KK[T_{i_1, \ldots, i_d}, S_l]$ for short in the following.
Let $I_{d,n}$ denote the ideal of
Plücker relations regarded in the above polynomial ring and consider the factor ring
$$
R:=\KK[T_{i_1, \ldots, i_d}, S_l]/I_{d,n}. 
$$
Let $K$ be a finitely generated abelian group and fix any $K$-grading on $\KK[T_{i_1, \ldots, i_d}, S_l]$ such 
that the variables $T_{i_1, \ldots, i_d}$ and $S_l$ as
well as the Plücker relations are homogeneous. 
We assume the grading to 
fulfill the following conditions:
The grading must be \emph{pointed}, i.e.\, we have $R_0= \KK$ and the \emph{effective cone}
$$
\mathrm{cone}(w_{i_1, \ldots, i_d}, w_l; \ 1 \leq i_1 < \ldots < i_d \leq n, 1 \leq l \leq m)
$$
contains no lines.
Furthermore the grading has to be \emph{almost free}, i.e.\, any 
$m + \binom{n}{d} -1$ of the weights
$w_{i_1, \ldots, i_d}$ and $w_l$ generate $K$ as a group, and the \emph{moving cone}
$$
\bigcap_{\gamma_0 \preceq \gamma \text{ facet }} \mathrm{cone}(w_{i_1, \ldots, i_d}, w_l; \ 
e_{i_1, \ldots, i_d}, e_l \in \gamma_0)
\quad
\text{where}
\quad
\gamma := \QQ_{\geq 0}^{\binom{n}{d} + m}
$$
must be full-dimensional. 
Let $u$ be any element in the 
relative interior of the moving cone. 
Then we obtain a commutative diagram
$$
\xymatrix{
\mathrm{Spec}\, R
\ar@{}[r]|{\qquad =}
&
{\bar{X}}
\ar[r]
\ar@{}[d]|{\rotatebox[origin=c]{90}{$\scriptstyle\subseteq$}}
&
{\bar{Z}}
\ar@{}[r]|{= \quad}
\ar@{}[d]|{\rotatebox[origin=c]{90}{$\scriptstyle\subseteq$}}
&
{\KK^{\binom{n}{d} +m}}
\\
&
{\overline{X}^{ss}(u)} 
\ar[r]
\ar[d]_{/ \! \! / H}
& 
{\overline{Z}^{ss}(u)}
\ar[d]^{/ \! \! / H}
&
\\
&
X
\ar[r]
&
Z
}
$$ 
where $H := \Spec\, \KK[K]$ is a quasitorus, the vertical arrows are GIT quotients corresponding to $u$ and the horizontal arrows are closed embeddings.
The varieties $X$ and $Z$ are normal projective varieties and we have
$$
\mathrm{Cl}(X) = \mathrm{Cl}(Z) \cong K
\quad 
\text{and}
\quad  
\mathcal{R}(X) = R,
$$
for the divisor class group and the Cox ring of $X$.
We call $X$ an \emph{intrinsic Grassmannian of type $(d,n)$}
and if  $m = 0$  we call $X$ a \emph{full intrinsic Grassmannian of type $(d,n)$}.
\end{construction}
\begin{proof}[Proof of Construction \ref{constr:intrinsGrassmannian}]
According to \cite{ArDeHaLa}*{Cor. 1.6.4.4} it suffices to prove that the 
$K$-grading installed on $R$ is factorial, meaning that we have unique factorization in the monoid of non-zero homogeneous elements $R^\times$ and that the variables $T_{i_1, \ldots, i_d}$ and $S_l$ are $K$-prime, i.e.\ prime in $R^\times$. This indeed holds as $R$ is a UFD and the variables are even prime. 
\end{proof}

\begin{remark}
In the notation of Construction \ref{constr:intrinsGrassmannianDN}
the variety $Z$ is a toric variety of dimension
$\mathrm{dim}(Z) = \binom{n}{d} + m - \mathrm{dim}(K_\QQ)$.
The dimension of $X$ is given as
$$
\mathrm{dim}(X) =  d(n-d)  + m - \mathrm{dim}(K_\QQ) + 1.
$$
\end{remark}

From now on we will restrict ourselves to the case of intrinsic Grassmannians of type $(2,n)$. Let us fix the notation:

\begin{notation}[Intrinsic Grassmannians of type $(2,n)$]
\label{constr:intrinsGrassmannian}
Fix integers $n \geq 4$ and $m \geq 0$.
As above we set
$$
\KK[T_{ij}, S_l]:= \KK[T_{ij}, S_l;\ 1 \leq i < j \leq n, 1 \leq l \leq m].
$$
for short. 
For $I = \left\{a,b,c,d\right\}$ with
$1 \leq a < b < c < d \leq n$, we  set
$$
g_I := 
T_{ab} T_{cd} - T_{ac}T_{bd} + T_{ad}T_{bc}
,$$
and denote the factor ring with
$$
R(n,m):=\KK[T_{ij}, S_l]/I_{2,n}, 
\quad
\text{where}
\quad
I_{2,n} =\bangle{g_I; \  I \subseteq \left\{1, \ldots, n\right\}, |I| = 4}.
$$
Now, fix any $K$-grading
on $\KK[T_{ij}, S_l]$ 
as in Construction ~\ref{constr:intrinsGrassmannianDN}. We~set 
$$
Q:= [w_{12},w_{13}, \ldots, w_{(n-1)n}, w_1, \ldots, w_m]
$$
for the matrix whose columns are the lexicographically ordered generator degrees 
$w_{ij} := \deg(T_{ij}) \in K$
followed by the generator degrees $w_l := \deg(S_l) \in K$
and denote
the ring $R(n,m)$ endowed with the above grading
with $R(n,m,Q)$. 
Note that in this case any grading that leaves the variables and a generating 
set of $I_{2,n}$ homogeneous, automatically leaves the relations $g_I$ homogeneous.
\end{notation}

The realization of an intrinsic Grassmannian $X$ of type $(2,n)$ as a geometric invariant theory quotient enables us to give the following explicit formula for the anticanonical class of $X$:

\begin{proposition}
\label{prop:AnticanClass}
Let $X$ be an intrinsic Grassmannian of type $(2,n)$
specified by its $K$-graded Cox ring $R(n,m,Q)$
and an ample class $u \in K = \Cl(X)$.
Then the anticanonical class is given as \begin{equation}\label{eq:mg_can}- \mathcal{K}_X
=\left(\frac{2}{n-1}\right)\sum_{1\leq i < j \leq n} w_{ij} + \sum_{i=1}^m w_l\in K = \mathrm{Cl}(X)\end{equation}
\end{proposition}

In order to prove the above proposition, 
we will compare the $H$-action on $\mathrm{Spec}\, R(n,0) = \mathrm{aGr}(2,n)$ defined in
Construction~\ref{constr:intrinsGrassmannian} with the 
following one; see~\cite{wg}:

\begin{remark}\label{rem:symmtryGroup}
Consider the slightly modified Plücker embedding
$\mathrm{aGr}(2,n) \subseteq \bigwedge^2 \KK^n \otimes \KK$.
Then $\mathrm{GL}(n; \KK)$ acts on the first factor and $\KK^*$ acts on the second by scalar multiplication.
These actions leave $\mathrm{aGr}(2,n)$ invariant. Now let $H \subseteq \TT^n$ be any sub-quasitorus of the $n$-dimensional standard torus and fix elements $\chi^{w_1}, \ldots, \chi^{w_n}, \chi^u$ in its character group $\Chi(H)$. 
We consider the morphism
$$
H \rightarrow \mathrm{Gl}(n; \KK) \times \KK^*
\qquad
t \mapsto 
\left(
\begin{pmatrix}
\chi^{w_1}(t) & & 
\\
& \ddots & 
\\
& & \chi^{w_n}(t)
\end{pmatrix},
\chi^u(t)
\right).
$$
Then the composition with
the action of $\mathrm{GL}(n; \KK) \times \KK^*$ on $\bigwedge^2 \KK^n \otimes \KK$
gives rise to an $H$-action on $\mathrm{aGr}(2,n)$.
The corresponding $\Chi(H)$-grading on 
$R(n,0)$ is given via
$\mathrm{deg}(T_{ij}) = w_i + w_j + u$.
\end{remark}

\begin{remark}
Let $R(n,0)$ be as in Construction \ref{constr:intrinsGrassmannian}
and fix an effective $K$-grading by a finitely generated abelian group $K$ that leaves the variables $T_{ij}$ as well as 
the relations $g_I$ homogeneous.
Then $K$ is isomorphic to a factor group of $\ZZ^n$. In particular $H := \Spec \, \KK[K]$ can be realized as a subtorus of the standard $n$-torus.
\end{remark}

\begin{lemma}\label{lem:grading}
Let $R(n,0)$ be as in Construction~\ref{constr:intrinsGrassmannian} and fix a finitely generated abelian group $K$. Then any $K$-grading that leaves the variables $T_{ij}$ as well as the relations $g_I$ homogeneous arises via 
an $H:=\mathrm{Spec} \, \KK[K]$-action as constructed in Remark~\ref{rem:symmtryGroup}.
\end{lemma}
\begin{proof}
We fix any $K$-grading on $R(n,0)$ that leaves the variables $T_{ij}$ as well as the relations $g_I$ homogeneous
and set $w_{ij} := \mathrm{deg}(T_{ij})$.
Furthermore for $1 \leq k \leq n$ set 
$\tilde{w}_k := w_{kj} + w_{ik} - w_{ij}$,
where $1 \leq i,j\leq n$ and $i,j,k$ are pairwise different. Here, we identify $w_{ij}$ with $w_{ji}$ if necessary. 
We note that $\tilde{w}_k$ does not depend on the choice of $i$ and $j$ as for pairwise different integers
$1 \leq i_1,i_2,j_1,j_2,k \leq n$ homogeneity of the relations implies
\begin{equation}\label{equ4}
   w_{kj_1} + w_{i_1k} - w_{i_1j_1}  
= w_{kj_2} + w_{i_2k} - w_{i_2j_2}. 
\end{equation}
Choose $w_1, u \in K$
such that $\tilde{w}_1 = 2 w_1 + u$.
Then using (\ref{equ4}), for any $2 \leq k \leq n$ we obtain $w_k \in K$ with $\tilde{w}_k = 2 w_k + u$.
Then for $1 \leq i < j \leq n$ we obtain
$$
2 (w_i + w_j + u) = \tilde{w}_i + \tilde{w}_j = 
(w_{ij} + w_{kj} - w_{ik}) + (w_{ij}+w_{ik}-w_{kj}) = 2w_{ij},
$$
where $1 \leq k\leq n$ with $k \notin \left\{i,j\right\}$.
In particular, the $K$-grading defined via the weights $w_{ij}$ corresponds to
the quasitorus action 
defined by 
$$
\mathrm{Spec}\, \KK[K] \rightarrow \mathrm{GL}(n; \KK) \times \KK^*
\qquad
t \mapsto 
\left(
\begin{pmatrix}
\chi^{w_1}(t) & & 
\\
& \ddots & 
\\
& & \chi^{w_n}(t)
\end{pmatrix},
\chi^u(t)
\right)
$$
composed with
the action of $\mathrm{GL}(n; \KK) \times \KK^*$ on $\bigwedge^2 \KK^n \otimes \KK$.
\end{proof}

\begin{proof}[Proof of Proposition \ref{prop:AnticanClass}]
First we deal with the case that $X$ is a full intrinsic Grassmannian, i.e.\ we have $m=0$. 
Consider the embedding 
$\mathrm{aGr}(2,n) \subseteq \bigwedge^2 \KK^n \otimes \KK$
as introduced in Remark \ref{rem:symmtryGroup}.
We set $V$ for the vectorspace $\KK^n$ endowed with the natural action of $\mathrm{GL}(n;\KK)$
and $L$ for $\KK$ endowed with the natural action of $\KK^*$. 
Furthermore set 
$$
D:= \bigwedge^n V \otimes L^2
$$
with the induced action.
Then following \cite{wg}*{Sec. 2} and \cite{Weyman-Book}*{Sec. 6.4} we obtain
a $\mathrm{GL}(n; \KK) \times \KK^*$-equivariant
projective resolution 
of length \(c := \mathrm{codim}(\mathrm{aGr}(2,n))\)
of the 
structure sheaf $\mathcal{O}_{\mathrm{aGr}(2,n)}$
\begin{equation}
\label{eq:res}
0 \leftarrow\ \OOO\leftarrow \mathcal{F}_1\leftarrow \cdots\leftarrow \mathcal F_c \leftarrow 0 
\qquad
\text{with}
\qquad
\mathcal{F}_c = \mathcal{O} \otimes D^{n-3} \otimes L^{\binom{n-3}{2}},
\end{equation}
where $\mathcal{O}$ denotes the structure sheaf of $\KK^{\binom{n}{2}}$.
Now, consider a geometric invariant quotient $X \subseteq Z$ 
of $\mathrm{aGr}(2,n)$ with respect to a quasitorus $H$ as constructed in \ref{constr:intrinsGrassmannian}.
Then due to Lemma~\ref{lem:grading}  the $H$-action on $\mathrm{aGr}(2,n)$ can be realized as a subquasitorus-action of the action of $\mathrm{GL}(n; \KK) \times \KK^*$. 
In particular, we can read the anticanonical class off the above resolution:
\begin{align*}
\Chi(H) = \mathrm{Cl}(X) \ni - \mathcal{K}_X
&= 
\mathcal{K}_Z
-
\mathrm{deg}(D^{n-3} \otimes L^{\binom{n-3}{2}})
\\
&=
\left(
1 -  \frac{n-3}{n-1}\right)
\sum
w_{ij}
= 
\left(\frac{2}{n-1} \right)
\sum
w_{ij}.
\end{align*}
Now assume $m > 1$. 
As the free variables $S_l$, where $1\le l\le m$ of degrees \(w_1,\ldots,w_m\) are not involved in the defining relations of \(\aGr(2,n)\) the  free resolution remains the same and we get the required formula \eqref{eq:mg_can}.
\end{proof}

\begin{rmk} 
The canonical bundle formulae  of \(\w\Gr(2,5)\) in \cite{wg} and of \(\w\Gr(2,6)\) in \cite{QS} are special cases of  \eqref{eq:mg_can} with \(K=\ZZ\) and \(m=0\).\end{rmk}

In order to obtain a smoothness criterion for an intrinsic Grassmannian $X$ of type $(2,n)$ we cut down the orbits of the ambient toric variety $Z$ to the variety $X$:

\begin{construction}\label{constr:relevantFaces}
Let $X \subseteq Z$ arise via Construction \ref{constr:intrinsGrassmannian}, let ${\gamma := \QQ_{\geq0}^{\binom{n}{2} + m}}$ denote
the positive orthant 
with rays $\gamma_{ij} := \mathrm{cone}(e_{ij})$
and $\gamma_l := \mathrm{cone}(e_l)$
and consider the linear map
$$
Q \colon \ZZ^{\binom{n}{2} + m} \rightarrow K, \quad 
e_{ij} \mapsto w_{ij}, \ e_{l} \mapsto w_{l}.
$$
Let $\gamma_0 \preceq \gamma$ be any face and denote with $\overline{Z}(\gamma_0)$ the set of all points $z \in \overline{Z}$ 
with coordinates $z_{ij} \neq 0$ resp. $z_l\neq 0$ whenever $e_{ij} \in \gamma_0$ resp. $e_l \in \gamma_0$. We call a face $\gamma_0 \preceq \gamma$ a \emph{$Z$-relevant face} if $Q(\gamma_0) \subseteq K_\QQ$ contains $u$ in its relative interior. 
We have 
$$\overline{Z}^{ss}(u) = \bigcup_{\substack{\gamma_0 \preceq \gamma_1 \preceq \gamma \\
\gamma_0 \text{ $Z$-relevant}}} \overline{Z}(\gamma_1).
$$
Moreover, let $\pi \colon \overline{Z}^{ss}(u) \rightarrow Z$ denote the quotient map. Then the images $Z(\gamma_0) := \pi(Z(\gamma_0))$, where $\gamma_0 \preceq \gamma$ is a $Z$-relevant face, are precisely the torus orbits of $Z$. 

Now, we call $\gamma_0\preceq \gamma$ an \emph{$\overline{X}$-face} if 
$
\overline{X}(\gamma_0) := \overline{Z}(\gamma_0) \cap \overline{X}
$
is non-empty and say $\gamma_0$ is \emph{$X$-relevant} if furthermore it is $Z$-relevant. We write $\mathrm{rlv}(X)$ for the set of $X$-relevant faces for short. Note that the $X$-relevant faces correspond to those torus orbits of $Z$ that intersect $X$ non-trivially. 
For a relevant face $\gamma_0 \preceq \gamma$, we set $X(\gamma_0) := Z(\gamma_0) \cap X \subseteq X$ and call these subsets the \emph{pieces of $X$}. 
\end{construction}

\begin{lemma}\label{lem:XbarFaces}
Let $X \subseteq Z$ be an intrinsic Grassmannian of type $(2,n)$ 
as in Construction \ref{constr:intrinsGrassmannian} and let $\gamma :=  \QQ_{\geq 0}^{\binom{n}{2} + m}$ denote the
positive orthant with rays $\gamma_{ij} := \mathrm{cone}(e_{ij})$
and $\gamma_l := \mathrm{cone}(e_l)$.
Set furthermore
$$
\gamma_{l_1,l_2} := \gamma_{l_1} + \gamma_{l_2},
\quad
\gamma_{l, i_1j_1} := \gamma_l + \gamma_{i_1j_1}
\quad
\text{and}
\quad
\gamma_{i_1j_1, i_2j_2} := \gamma_{i_1j_1} + \gamma_{i_2j_2}
$$
Then the following statements hold:
\begin{enumerate}
    \item 
    All faces $\gamma_l$, $\gamma_{l_1,l_2}$, $\gamma_{ij}$ 
    and  $\gamma_{l,ij}$ are $\overline{X}$-faces.
    \item
    A face $\gamma_{i_1j_1, i_2j_2}$ with $(i_1,j_1) \neq (i_2,j_2)$ is an $\overline{X}$-face
    if and only if $|\left\{i_1, i_2, j_1, j_2\right\}| = 3.$
    \item
    Let $\gamma_0$ be any $\overline{X}$-face. 
    Then $\overline{X}(\gamma_0)$ consists of smooth points if and only if there exists at least one $\gamma_{ij}$ with
    $\gamma_{ij} \preceq \gamma'$.
    \item
    The piece $X(\gamma_0)$ associated to an $X$-relevant face $\gamma_0 \preceq \gamma$ consists of smooth points of $X$ if and only if the following two conditions hold:
    \begin{enumerate}
        \item 
        $Q(\mathrm{lin}_{\QQ}(\gamma_0) \cap \ZZ^{\binom{n}{2} + m})$
        generates $K$ as a group.
        \item
        There exists at least one $\gamma_{ij}$ with $\gamma_{ij} \preceq \gamma_0$.
    \end{enumerate}
\end{enumerate}
\end{lemma}
\begin{proof}
Assertion~(i) follows directly from the definition of an $\overline{X}$-face. 
We prove (ii). A face $\gamma_{i_1j_1, i_2j_2}$ is an $\overline{X}$-face if and only if there exists no term
of the form $T_{i_1j_1}T_{i_2j_2}$ in any of the relations $g_I$. The latter is equivalent to $|\left\{i_1, i_2, j_1, j_2\right\}| = 3$.
Assertion (iii) follows from the fact that $0$ is the only singular point of $\mathrm{aGr}(2,n)$. Assertion (iv) follows from (iii) and 
\cite{ArDeHaLa}*{Prop. 3.3.1.10}
\end{proof}

\begin{remark}\label{rem:effAmpleEtc}
Let $X\subseteq Z$ be an intrinsic Grassmannian of type $(2,n)$ as in Construction \ref{constr:intrinsGrassmannian}. 
Then in $K = \mathrm{Cl}(X)$ the Picard group is given as
$$
\mathrm{Pic}(X) =
\bigcap_{\gamma_0 \preceq \gamma \text{ relevant}}
Q( \mathrm{lin}(\gamma_0) \cap \ZZ^{\binom{n}{2} + m})
$$
Moreover, due to \cite{ArDeHaLa}*{Prop. 3.3.2.9} the
cones of effective, movable, semiample and ample divisor classes in $K_\QQ = \mathrm{Cl}(X)_\QQ$ are given as
$$
\mathrm{Eff}(X) = Q(\gamma),
\qquad
\mathrm{Mov}(X) = \bigcap_{\gamma_0 \preceq \gamma \text{ facet}}
Q(\gamma_0),
$$
$$
\mathrm{SAmple}(X) = \bigcap_{\gamma_0 \preceq \gamma \text{ relevant}} Q(\gamma_0),
\qquad
\mathrm{Ample}(X) = \bigcap_{\gamma_0 \preceq \gamma \text{ relevant}} Q(\gamma_0)^\circ,
$$
where $Q(\gamma_0)^\circ$ denotes the relative interior of $Q(\gamma_0)$.
Moreover, the variety $X$ is $\QQ$-factorial if and only if the semiample cone
$\mathrm{SAmple}(X)$ is of full dimension in $K_\QQ$.
\end{remark}

\section{Classification in Picard Number two}\label{sec:classification}
In this section we provide
a complete classification of  smooth intrinsic Grassmannians of type $(2,n)$ with Picard number two; see Theorem~\ref{thm:classNonFull}.
Moreover, in Corollary \ref{cor:FanoNonFull} we obtain criteria for 
these varieties 
to be Fano, by using our description of the anticanonical class in 
Proposition~\ref{prop:AnticanClass}. 
Note that these results generalize
Theorem~\ref{thm:classification} and Corollary~\ref{cor:FanoClassification} stated in the introduction, which treat the case of \emph{full} intrinsic Grassmannians of type~$(2,n)$.

\begin{theorem}\label{thm:classNonFull}
Let $X$ be a smooth intrinsic Grassmannian of type $(2,n)$ with Picard number $\rho(X) =2$. Then $X$ has divisor class group $\mathrm{Cl}(X) = \ZZ^2$ and 
the $\mathrm{Cl}(X)$-grading of $\mathcal{R}(X) = R(n,m)$ and the semiample cone are of one of the six types below;
we write $w_{ij} := \mathrm{deg}(T_{ij})$
and $w_l = \mathrm{deg}(S_l)$
for the $\mathrm{Cl}(X)$-degrees.

\vspace{1mm}
\noindent
\textbf{Type 1:} 
Fix integers $4 \leq k \leq n$, $m \geq 0$ and $a \geq 0$. 
Moreover fix sequences of integers
$ 0 \leq  \alpha_k \leq \alpha_{k+1} \leq \ldots \leq \alpha_n\leq a$
and
$ 0 \leq \beta_1 \leq \ldots \leq \beta_m \leq a$.
We have
\begin{enumerate}
    \item
    $w_l = (\beta_l,1)$ for $1 \leq l \leq m$,
    \item 
    $w_{ij} = (1,0)$, whenever $j < k$ holds,
    \item
    $w_{ij} = (\alpha_j, 1)$, whenever $i < k \leq j$ holds,
    \item 
    $w_{ij} = (\alpha_i + \alpha_j -1, 2)$, whenever $k \leq i$ holds, and
    \item
    $(0,1)$ and $(a,1)$ occur among the weights $w_{ij}$ and $w_l$.
\end{enumerate}
The semiample cone looks as follows:

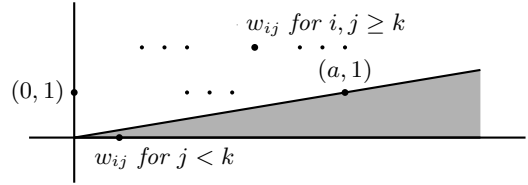
\begin{center}
\begin{tikzpicture}[scale=0.6]
%semiample cone
\path[fill=gray!60!] (0,0)--(9,1.5)--(9,0);
%rays of the semiample cone \tau_X
\draw[-,thick] (0,0)--(9,1.5);
\draw[-,thick] (0,0)--(9,0);
%Cl(X)-degrees of type (i)
\path[fill, color=black] (1,0) circle (0.5ex);
\path[fill, color=black] (2,0) circle (0.0ex)
node[below]{\small $w_{ij}$ for $j < k$};
%Cl(X)-degrees of type (ii)
\path[fill, color=black] (6,1) circle (0.5ex) 
node[above]{\small $(a,1)$};
\path[fill, color=black] (0,1) circle (0.5ex)
node[left]{\small $(0,1)$};
\path[fill, color=black] (2.5,1) circle (0.3ex);
\path[fill, color=black] (3,1) circle (0.3ex);
\path[fill, color=black] (3.5,1) circle (0.3ex);
%Cl(X)-degrees of type (iii) for i,j \geq k
\path[fill, color=black] (4,2) circle (0.5ex);
\path[fill, color=black] (3.6,2.5) circle (0.0ex)
node[right]{\small $w_{ij}$ for $i,j\geq k$};
\path[fill, color=black] (1.5,2) circle (0.3ex);
\path[fill, color=black] (2,2) circle (0.3ex);
\path[fill, color=black] (2.5,2) circle (0.3ex);
\path[fill, color=black] (5,2) circle (0.3ex);
\path[fill, color=black] (5.5,2) circle (0.3ex);
\path[fill, color=black] (6,2) circle (0.3ex);
%axes
\draw[-,thick] (-1,0)--(10,0);
\draw[-,thick] (0,-1)--(0,3);
\end{tikzpicture}   
\end{center}
%%%%%%%%%%%%%%%%%%%%%%%%%%%%%%%%%%%%%
%%%%%%%%%%%%%%%%%%%%%%%%%%%%%%%%%%%%%
%%%%%%%%%%%%%%%%%%%%%%%%%%%%%%%%%%%%%

\vspace{1mm}
\noindent
\textbf{Type 2:} 
Fix integers
$m \geq 0$ and $a \geq 0$. Let furthermore $0 \leq \alpha \leq a$
and
$0 \leq \beta_1 \leq \ldots \leq \beta_m \leq a$ be integers.
We have
\begin{enumerate}
    \item
    $w_l = (\beta_l,1)$ for $1 \leq l \leq m$,
    \item 
    $w_{ij} = (\alpha,1)$ whenever $j < n$ holds,
    \item
    $w_{in} = (1, 0)$ for $1 \leq i \leq n-1$, and
    \item
    $(0,1)$ and $(a,1)$ occur among the weights $w_{ij}$ and $w_l$.
\end{enumerate}
The semiample cone looks as follows:

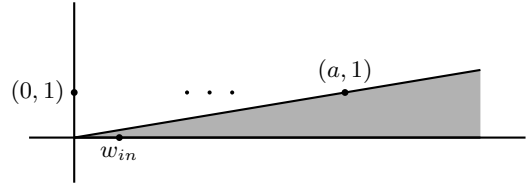
\begin{center}
\begin{tikzpicture}[scale=0.6]
%semiample cone
\path[fill=gray!60!] (0,0)--(9,1.5)--(9,0);
%rays of the semiample cone \tau_X
\draw[-,thick] (0,0)--(9,1.5);
\draw[-,thick] (0,0)--(9,0);
%Cl(X)-degrees of type (iii)
\path[fill, color=black] (1,0) circle (0.5ex)
node[below]{\small $w_{in}$};
%Cl(X)-degrees of type (ii)
\path[fill, color=black] (6,1) circle (0.5ex) 
node[above]{\small $(a,1)$};
\path[fill, color=black] (0,1) circle (0.5ex)
node[left]{\small $(0,1)$};
\path[fill, color=black] (2.5,1) circle (0.3ex);
\path[fill, color=black] (3,1) circle (0.3ex);
\path[fill, color=black] (3.5,1) circle (0.3ex);
%axes
\draw[-,thick] (-1,0)--(10,0);
\draw[-,thick] (0,-1)--(0,3);
\end{tikzpicture}   
\end{center}

%%%%%%%%%%%%%%%%%%%%%%%%%%%%%%%%%%%%%
%%%%%%%%%%%%%%%%%%%%%%%%%%%%%%%%%%%%%
%%%%%%%%%%%%%%%%%%%%%%%%%%%%%%%%%%%%%

\vspace{1mm}
\noindent
\textbf{Type 3:} 
Fix integers $4 \leq k < n$ and
$m \geq 0$.
We have
\begin{enumerate}
    \item
    $w_1 =  \ldots = w_m = (1,0)$,
    \item 
    $w_{ij} = (2,1)$, whenever $j < k$ holds,
    \item
    $w_{ij} = (0,1)$, whenever
    $k \leq i$.
    \item 
    $w_{ij} = (1,1)$ else.
\end{enumerate}
\begin{comment}
\begin{enumerate}
    \item
    $w_i = (1,-1)$ for $1 \leq i \leq m$.
    \item 
    $w_{ij} = (1,0)$, whenever $j < k$ holds,
    \item
    $w_{ij} = (-1,2)$, whenever
    $i \geq k$.
    \item 
    $w_{ij} = (0,1)$ else.
\end{enumerate}

\end{comment}
The semiample cone looks as follows:

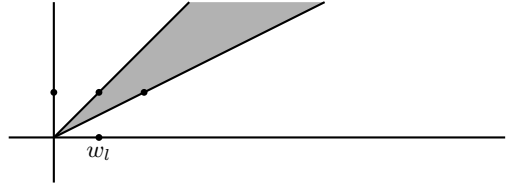
\begin{center}
\begin{tikzpicture}[scale=0.6]
%semiample cone
\path[fill=gray!60!] (0,0)--(6,3)--(3,3);
%rays of the semiample cone \tau_X
\draw[-,thick] (0,0)--(6,3);
\draw[-,thick] (0,0)--(3,3);
%Cl(X)-degrees of type (i)
\path[fill, color=black] (1,0) circle (0.5ex)
node[below]{\small $w_l$};
%Cl(X)-degrees of type (ii)
\path[fill, color=black] (2,1) circle (0.5ex);
%node[above]{\small $w_i$};
%Cl(X)-degrees of type (iii)
\path[fill, color=black] (0,1) circle (0.5ex);
%node[left]{\small $(\alpha_k,1)$};
%Cl(X)-degrees of type (iv)
\path[fill, color=black] (1,1) circle (0.5ex);
%node[left]{\small $(\alpha_k,1)$};
%axes
\draw[-,thick] (-1,0)--(10,0);
\draw[-,thick] (0,-1)--(0,3);
\end{tikzpicture}   
\end{center}

%%%%%%%%%%%%%%%%%%%%%%%%%%%%%%%%%%%%%
%%%%%%%%%%%%%%%%%%%%%%%%%%%%%%%%%%%%%
%%%%%%%%%%%%%%%%%%%%%%%%%%%%%%%%%%%%%

\vspace{1mm}
\noindent
\textbf{Type 4:} 
Fix an integer $m\geq 1$. We have
\begin{enumerate}
    \item $w_1 = \ldots = w_m = (1,0)$ and $w_{12} = (2,1)$
    \item 
    $w_{ij} = (1,1)$, whenever $i = 1,2$ holds, and
    \item
    $w_{ij} = (0,1)$, whenever $i \geq 3$ holds.
\end{enumerate}
The semiample cone looks as follows:

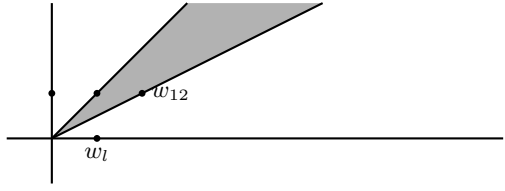
\begin{center}
\begin{tikzpicture}[scale=0.6]
%semiample cone
\path[fill=gray!60!] (0,0)--(6,3)--(3,3);
%rays of the semiample cone \tau_X
\draw[-,thick] (0,0)--(6,3);
\draw[-,thick] (0,0)--(3,3);
%Cl(X)-degrees of type (i)
\path[fill, color=black] (1,0) circle (0.5ex)
node[below]{\small $w_l$};
%Cl(X)-degrees of type (ii)
\path[fill, color=black] (2,1) circle (0.5ex)
node[right]{\small $w_{12}$};
%node[above]{\small $w_i$};
%Cl(X)-degrees of type (iii)
\path[fill, color=black] (0,1) circle (0.5ex);
%node[left]{\small $(\alpha_k,1)$};
%Cl(X)-degrees of type (iv)
\path[fill, color=black] (1,1) circle (0.5ex);
%node[left]{\small $(\alpha_k,1)$};
%axes
\draw[-,thick] (-1,0)--(10,0);
\draw[-,thick] (0,-1)--(0,3);
\end{tikzpicture}   
\end{center}
%%%%%%%%%%%%%%%%%%%%%%%%%%%%%%%%%%%%%
%%%%%%%%%%%%%%%%%%%%%%%%%%%%%%%%%%%%%
%%%%%%%%%%%%%%%%%%%%%%%%%%%%%%%%%%%%%
%%%%%%%%%%%%%%%%%%%%%%%%%%%%%%%%%%%%%

\vspace{1mm}
\noindent
\textbf{Type 5:} 
Fix an integer $m\geq 2$ and
a sequence of integers
$0 = b_3 \leq  b_2 \leq b_1 \leq \alpha_4 \leq \ldots \leq \alpha_n$. We have $w_{ij} =(x_{ij}, 1)$ with $x_{ij} \geq 0$
for all $(i,j)$
and 
\begin{enumerate}
    \item 
    $w_1 = \ldots = w_m = (1,0)$,
    \item 
    $w_{12} = (b_3,1) = (0,1)$, $w_{13} =(b_2,1)$ and $w_{23} = (b_1,1)$,
    \item 
    $w_{ij} = (\alpha_j-b_i,1)$
    for all $i \leq 3 < j$,
    \item
    $w_{ij} = 
    (\alpha_i + \alpha_j - b_1 - b_2, 1)$
    else.
    %w_{1i} + w_{2j} - (0,1)$ 
\end{enumerate}
The semiample cone looks as follows:

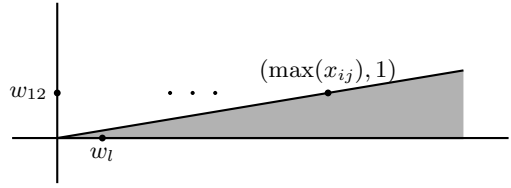
\begin{center}
\begin{tikzpicture}[scale=0.6]
%semiample cone
\path[fill=gray!60!] (0,0)--(9,1.5)--(9,0);
%rays of the semiample cone \tau_X
\draw[-,thick] (0,0)--(9,1.5);
\draw[-,thick] (0,0)--(9,0);
%Cl(X)-degrees of type (i)
\path[fill, color=black] (1,0) circle (0.5ex)
node[below]{\small $w_l$};
%Cl(X)-degrees of type (ii)
\path[fill, color=black] (6,1) circle (0.5ex) 
node[above]{\small $(\mathrm{max}(x_{ij}),1)$};
\path[fill, color=black] (0,1) circle (0.5ex)
node[left]{\small $w_{12}$};
\path[fill, color=black] (2.5,1) circle (0.3ex);
\path[fill, color=black] (3,1) circle (0.3ex);
\path[fill, color=black] (3.5,1) circle (0.3ex);
%axes
\draw[-,thick] (-1,0)--(10,0);
\draw[-,thick] (0,-1)--(0,3);
\end{tikzpicture}   
\end{center}

%%%%%%%%%%%%%%%%%%%%%%%%%%%%%%%%%%%%%
%%%%%%%%%%%%%%%%%%%%%%%%%%%%%%%%%%%%%
%%%%%%%%%%%%%%%%%%%%%%%%%%%%%%%%%%%%%
%%%%%%%%%%%%%%%%%%%%%%%%%%%%%%%%%%%%%
\vspace{1mm}
\noindent
\textbf{Type 6:} 
Fix an integer $m\geq 2$ and
a sequence of integers
$0 =  \beta_1 \leq \ldots \leq \beta_m$. We have
\begin{enumerate}
    \item $w_l = (\beta_l,1)$
    for $1 \leq l \leq m$, and
    \item 
    $w_{ij} = (1,0)$ for all $i,j$.
\end{enumerate}
The semiample cone looks as follows:

\begin{center}
\begin{tikzpicture}[scale=0.6]
%semiample cone
\path[fill=gray!60!] (0,0)--(9,1.5)--(9,0);
%rays of the semiample cone \tau_X
\draw[-,thick] (0,0)--(9,1.5);
\draw[-,thick] (0,0)--(9,0);
%Cl(X)-degrees of type (ii)
\path[fill, color=black] (1,0) circle (0.5ex)
node[below]{\small $w_{ij}$};
%Cl(X)-degrees of type (i)
\path[fill, color=black] (6,1) circle (0.5ex) 
node[above]{\small $(\beta_m,1)$};
\path[fill, color=black] (0,1) circle (0.5ex)
node[left]{\small $(\beta_1,1)$};
\path[fill, color=black] (2.5,1) circle (0.3ex);
\path[fill, color=black] (3,1) circle (0.3ex);
\path[fill, color=black] (3.5,1) circle (0.3ex);

%axes
\draw[-,thick] (-1,0)--(10,0);
\draw[-,thick] (0,-1)--(0,3);
\end{tikzpicture}   
\end{center}
\end{theorem}

\begin{corollary}\label{cor:FanoNonFull}
In the notation of Theorem \ref{thm:classNonFull}
the following two tables give criteria 
that give rise to the Fano and
truly almost Fano varieties among the smooth 
intrinsic Grassmannians of type $(2,n)$
with Picard number two:

{
\renewcommand*{\arraystretch}{2.5}
\begin{longtable}{c||c}
Type &
Fano
\\
\hline
\hline
1
&
$(n + \frac{m}{2} -k+1)a 
<

\sum\limits_{j = k}\limits^{n}\alpha_j
+ k - \frac{n}{2} -1
+ \frac{1}{2}\sum\limits_{l = 1 }\limits^m \beta_l$
\\
%%%%%%%%%%%%%%%%%%%%%%%%%%%%%%%%%%%%%%%%%%%%%%%%%%
%%%%%%%%%%%%%%%%%%%%%%%%%%%%%%%%%%%%%%%%%%%%%%%%%%
%%%%%%%%%%%%%%%%%%%%%%%%%%%%%%%%%%%%%%%%%%%%%%%%%%
\hline
2
&
$
(n+m-2)a
< 
\sum\limits_{l=1}\limits^m \beta_l + (n-2)\alpha + 2
$
\\
%%%%%%%%%%%%%%%%%%%%%%%%%%%%%%%%%%%%%%%%%%%%%%%%%%
%%%%%%%%%%%%%%%%%%%%%%%%%%%%%%%%%%%%%%%%%%%%%%%%%%
%%%%%%%%%%%%%%%%%%%%%%%%%%%%%%%%%%%%%%%%%%%%%%%%%%
\hline
3
&
$
1 < \frac{2(k-1) + m }{n} < 2
$
\\
%%%%%%%%%%%%%%%%%%%%%%%%%%%%%%%%%%%%%%%%%%%%%%%%%%
%%%%%%%%%%%%%%%%%%%%%%%%%%%%%%%%%%%%%%%%%%%%%%%%%%
%%%%%%%%%%%%%%%%%%%%%%%%%%%%%%%%%%%%%%%%%%%%%%%%%%
\hline
4
&
$
1 < \frac{4 + m}{n} < 2
$
\\
%%%%%%%%%%%%%%%%%%%%%%%%%%%%%%%%%%%%%%%%%%%%%%%%%%
%%%%%%%%%%%%%%%%%%%%%%%%%%%%%%%%%%%%%%%%%%%%%%%%%%
%%%%%%%%%%%%%%%%%%%%%%%%%%%%%%%%%%%%%%%%%%%%%%%%%%
\hline
5
&
$n \cdot \mathrm{max}(x_{ij}) < 
\sum\limits_{j\geq 4}2 \alpha_j
- (n-4)(b_1 + b_2) +m
$
\\
%%%%%%%%%%%%%%%%%%%%%%%%%%%%%%%%%%%%%%%%%%%%%%%%%%
%%%%%%%%%%%%%%%%%%%%%%%%%%%%%%%%%%%%%%%%%%%%%%%%%%
%%%%%%%%%%%%%%%%%%%%%%%%%%%%%%%%%%%%%%%%%%%%%%%%%%
\hline
6
&
$
1 < \frac{4 + m}{n} < 2
$
\\
\end{longtable}}
%%%%%%%%%%%%%%%%%%%%%%%%%%%%%%%%%%%%%%%%%%%%%%%%%%
%%%%%%%%%%%%%%%%%%%%%%%%%%%%%%%%%%%%%%%%%%%%%%%%%%
%%%%%%%%%%%%%%%%%%%%%%%%%%%%%%%%%%%%%%%%%%%%%%%%%%
%%%%%%%%%%%%%%%%%%%%%%%%%%%%%%%%%%%%%%%%%%%%%%%%%%
%%%%%%%%%%%% ALMOST FANO %%%%%%%%%%%%%%%%%%%%%%%%%
%%%%%%%%%%%%%%%%%%%%%%%%%%%%%%%%%%%%%%%%%%%%%%%%%%
%%%%%%%%%%%%%%%%%%%%%%%%%%%%%%%%%%%%%%%%%%%%%%%%%%
%%%%%%%%%%%%%%%%%%%%%%%%%%%%%%%%%%%%%%%%%%%%%%%%%%
%%%%%%%%%%%%%%%%%%%%%%%%%%%%%%%%%%%%%%%%%%%%%%%%%%

{
\renewcommand*{\arraystretch}{2.5}
\begin{longtable}{c||c}
Type &
Almost Fano
\\
\hline
\hline
1
&
$(n + \frac{m}{2} -k+1)a 
=
\sum\limits_{j = k}\limits^{n}\alpha_j
+ k - \frac{n}{2} -1
+ \frac{1}{2}\sum\limits_{l = 1 }\limits^m \beta_l$
\\
%%%%%%%%%%%%%%%%%%%%%%%%%%%%%%%%%%%%%%%%%%%%%%%%%%
%%%%%%%%%%%%%%%%%%%%%%%%%%%%%%%%%%%%%%%%%%%%%%%%%%
%%%%%%%%%%%%%%%%%%%%%%%%%%%%%%%%%%%%%%%%%%%%%%%%%%
\hline
2
&
$
(n+m-2)a
= 
\sum\limits_{l=1}\limits^m \beta_l + (n-2)\alpha + 2
$
\\
%%%%%%%%%%%%%%%%%%%%%%%%%%%%%%%%%%%%%%%%%%%%%%%%%%
%%%%%%%%%%%%%%%%%%%%%%%%%%%%%%%%%%%%%%%%%%%%%%%%%%
%%%%%%%%%%%%%%%%%%%%%%%%%%%%%%%%%%%%%%%%%%%%%%%%%%
\hline
3
&
$
\frac{2(k-1) + m }{n} 
= 1$
\quad
or
\quad
$
\frac{2(k-1) + m }{n}
= 2$
\\
%%%%%%%%%%%%%%%%%%%%%%%%%%%%%%%%%%%%%%%%%%%%%%%%%%
%%%%%%%%%%%%%%%%%%%%%%%%%%%%%%%%%%%%%%%%%%%%%%%%%%
%%%%%%%%%%%%%%%%%%%%%%%%%%%%%%%%%%%%%%%%%%%%%%%%%%
\hline
4
&
$
\frac{4 + m}{n} = 1
$
\quad
or
\quad
$
\frac{4 + m}{n} = 2
$
\\
%%%%%%%%%%%%%%%%%%%%%%%%%%%%%%%%%%%%%%%%%%%%%%%%%%
%%%%%%%%%%%%%%%%%%%%%%%%%%%%%%%%%%%%%%%%%%%%%%%%%%
%%%%%%%%%%%%%%%%%%%%%%%%%%%%%%%%%%%%%%%%%%%%%%%%%%
\hline
5
&
$n \cdot \mathrm{max}(x_{ij}) =
\sum_{j\geq 4}2 \alpha_j
- (n-4)(b_1 + b_2) +m
$
\\
%%%%%%%%%%%%%%%%%%%%%%%%%%%%%%%%%%%%%%%%%%%%%%%%%%
%%%%%%%%%%%%%%%%%%%%%%%%%%%%%%%%%%%%%%%%%%%%%%%%%%
%%%%%%%%%%%%%%%%%%%%%%%%%%%%%%%%%%%%%%%%%%%%%%%%%%
\hline
6
&
$
\frac{4 + m}{n} = 1 
$
\quad 
or
\quad
$
\frac{4 + m}{n} = 2
$
\\
\end{longtable}}
\end{corollary}

\section{Proofs to Sections \ref{sec:intro} and \ref{sec:classification}}
This section is dedicated to the proofs of Theorems \ref{thm:classification}, \ref{thm:classNonFull} and Corollaries
\ref{cor:FanoClassification}, \ref{cor:NumberOfFanos}
and
\ref{cor:FanoNonFull}.
We begin with discussing the effective cone and the $X$-relevant faces in the case of Picard number two.
Recall that a variety $X$ is called \emph{$\QQ$-factorial} if every Weil divisor has a multiple that is Cartier, and $X$ is called
\emph{locally factorial} if all its points are factorial, which is equivalent to the property that every Weil divisor of $X$ is Cartier

\begin{remark}\label{rem:conesAndCo}
Let $X$ be a $\QQ$-factorial intrinsic 
Grassmannian of type $(2,n)$ with Picard number two.
Then the effective cone has a decomposition 
$$
\mathrm{Eff}(X) = \tau^+ \cup \tau_X \cup \tau^-,
$$
where $\tau_X := \mathrm{Ample}(X)$ is the ample cone and $\tau^+$ and $\tau^-$ are closed cones not intersecting $\tau_X$ such that
$\tau^+ \cap \tau^- = \left\{0\right\}$ holds.

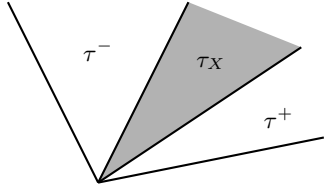
\begin{center}
\begin{tikzpicture}[scale=0.6]
%semiample cone
\path[fill=gray!60!] (0,0)--(4.5,3)--(2,4);
\draw (2.5,3) circle (0ex) node[below]{\small $\tau_X$};
%rays of the semiample cone \tau_X and \tau^+ and \tau^-
\draw[-,thick] (0,0)--(4.5,3);
\draw[-,thick] (0,0)--(2,4);
\draw[-,thick] (0,0)--(5,1);
\draw[-,thick] (0,0)--(-2,4);
\draw(4,1) circle (0ex) 
node[above]{\small $\tau^+$};
\draw(0,2.5) circle (0ex)
node[above]{\small $\tau^-$};
\end{tikzpicture}   
\end{center}
Consider the $\QQ$-linear map corresponding to the $\ZZ$-linear map $Q$ defined in Construction~\ref{constr:relevantFaces}:
$$
Q \colon \QQ^{\binom{n}{2} + m} \rightarrow K_\QQ, \quad
e_{ij} \mapsto w_{ij} \in K_\QQ, \ e_l \mapsto w_l \in K_\QQ.
$$
As the ample cone $\tau_X$ is a subset of the relative interior of $\mathrm{Mov}(X)$, we obtain that each of the cones $\tau^+$ and $\tau^-$ contains at least two of the weights of $w_{ij}$ or $w_l$. 
Moreover, as $X$ is $\QQ$-factorial $\tau_X$ is two-dimensional and thus none of the weights $w_{ij}$ and $w_l$ lie inside $\tau_X$ due to Lemma~\ref{lem:XbarFaces}~(i).
In particular, if $\gamma_0 \preceq \gamma$ is a
two-dimensional $\overline{X}$-face whose primitive ray generators are projected onto the weights $w$ and $w'$, then $\gamma_0$ is $X$-relevant if and only if $w \in \tau^-$ and $w' \in \tau^+$ holds or vice versa.
\end{remark}

\begin{lemma}
Let $X$ be a locally factorial intrinsic Grassmannian of type $(2,n)$ with Picard number two. Then $\Cl(X) = \mathrm{Pic}(X) = \ZZ^2$ holds. 
\end{lemma}
\begin{proof}
As $X$ is locally factorial we have $\mathrm{Pic}(X) = \Cl(X)$
and it suffices to show that $\mathrm{Pic}(X)$ is torsion free. Using Remark \ref{rem:effAmpleEtc} we obtain
$\mathrm{Pic}(X) \subseteq Q(\mathrm{lin}(\gamma_0) \cap \ZZ^{\binom{n}{2} + m})$, where $\gamma_0$ is a relevant face. In particular, in order to complete the proof we are left with finding a two-dimensional $X$-relevant face $\gamma_0 \preceq \gamma$.
We note that none of the weights $w_{ij}$ and $w_l$ lie in $\tau_X$.

Assume we have $m = 0$ and there exists no two-dimensional $X$-relevant face $\gamma_{i_1j_1, i_2j_2}$.
Then due to Lemma~\ref{lem:XbarFaces}~(ii)
either all weights lie in $\tau^+$ or in $\tau^-$.
This contradicts the fact that there are at least two weights in each of $\tau^+$ and $\tau^-$.

Now let $m \geq 1$ hold. We may assume $w_1 \in \tau^+$.
Due to Remark \ref{rem:conesAndCo} there is at least one weight $w_{ij}$ or $w_l$ that lies in $\tau^-$
and we obtain a two-dimensional $X$-relevant face $\gamma_{1, ij}$ or $\gamma_{1,l}$.
\end{proof}

\begin{remark}\label{rem:conesAndCo2}
Let $X$ be a smooth intrinsic Grassmannian of type $(2,n)$
with Picard number two. Then due to the above lemma we have $\Cl(X) = \ZZ^2$ and we write
$$
w_{ij} = (x_{ij}, y_{ij}) := \mathrm{deg}(T_{ij}) \in \ZZ^2
\qquad
w_{l} = (x_{l}, y_l) := \mathrm{deg}(S_l) \in \ZZ^2
$$
for the weights of the generators.
In the notation of Remark \ref{rem:conesAndCo}
consider any two-dimensional $X$-relevant face $\gamma_0 \preceq \gamma$
and denote with $w, w' \in K_\QQ$ the images of the primitive ray generators of $\gamma_0$ via $Q$. Then we may assume $w \in \tau^+$ and $w' \in \tau^-$ and due to Lemma \ref{lem:XbarFaces} we have 
$\mathrm{det}(w,w') = 1$. Thus applying a suitable coordinate change on $\ZZ^2$ we obtain 
$$w = (1,0)\quad \text{and} \quad w' = (0,1).$$
Moreover, in this situation we have $w'' = (x'',1) \in K_\QQ$ whenever $w, w''$ are the images of the primitive ray generators of
a two-dimensional $X$-relevant face via $Q$
and we have $w'' = (1, y'') \in K_\QQ$ whenever $w'', w'$ are the images of the primitive ray generators of
a two-dimensional $X$-relevant face via $Q$.
\end{remark}

\begin{lemma}\label{lem:Pic2Weights}
Let $X$ be a smooth intrinsic Grassmannian of type $(2,n)$ with Picard number two
arising via Construction \ref{constr:intrinsGrassmannian}.
Then the following statements hold:
\begin{enumerate}
    \item 
    The weights $w_l$, where $1 \leq l \leq m$, lie either all in $\tau^+$ or in~$\tau^-$.
    \item
    All minimal $X$-relevant faces $\gamma_0 \preceq \gamma$ with respect to inclusion are two-dimensional.
\end{enumerate}
\end{lemma}
\begin{proof}
We proof (i).
Due to Remark \ref{rem:conesAndCo}
a weight of the form $w_l$ lies either in $\tau^+$ or in $\tau^-$. Now assume we have $w_{l_1} \in \tau^+$ and $w_{l_2} \in \tau^-$. Then $\gamma_{l_1,l_2}$ is an $X$-relevant face but the corresponding piece $X(\gamma_{l_1,l_2})$ is singular due to Lemma~\ref{lem:XbarFaces}~(iv); a contradiction to smoothness of $X$.

We turn to (ii).
As $X$ is $\QQ$-factorial and of Picard number two there are no one-dimensional $X$-relevant faces.
Assume there is a minimal $X$-relevant face $\gamma_0$ of dimension at least three that does not contain any two-dimensional $X$-relevant face.
Using (i) and Lemma~\ref{lem:XbarFaces}~(i) this implies $e_l \notin \gamma_0$ for all $1 \leq l \leq m$. We obtain
$$
\gamma_0 = \mathrm{cone}(e_{i_1j_1}, \ldots, e_{i_kj_k}) \text{ for some } k \geq 3
$$
and after renumbering 
we may assume $w_{i_1j_1} \in \tau^+$
and $ w_{i_2j_2},w_{i_3j_3}\in \tau^-$.
As $\gamma_0$ does not contain any two-dimensional $X$-relevant face we conclude
$i_1, j_1 \notin \left\{i_2, j_2, i_3,j_3\right\}$.
This implies that
$T_{i_1j_1} T_{i_2j_2}$ occurs as a term in
one of the relations $g_I$. As
$\overline{X}(\gamma_0) \neq \emptyset$
holds
we conclude
\begin{equation}\label{equ2}
e_{i_1,i_2}, e_{j_1,j_2} \in \gamma_0
\quad 
\text{or}
\quad
e_{i_1, j_2}, e_{i_2, j_1} \in \gamma_0.
\end{equation}
Here, we identify $e_{ij}$ with $e_{ji}$ if necessary.
Assume the former holds and $w_{i_1i_2} \in \tau^+$ holds. Then $\gamma_{i_1i_2, i_2j_2}$ is a two-dimensional $X$-relevant face contained in $\gamma_0$; a contradiction.
So assume $w_{i_1i_2} \in \tau^-$. 
Then the two-dimensional $X$-relevant face $\gamma_{i_1j_1,i_1i_2}$ is contained in $\gamma_0$.
Again a contradiction. The same arguments work for 
the latter case in (\ref{equ2}) as well.
\end{proof}

\begin{remark}\label{rem:permutation}
Let $R:=R(n,m,Q)$ be any $K$-graded ring as in 
Construction~\ref{constr:intrinsGrassmannian}
and let $\sigma \in S_n$ be any permutation.
Then we have an isomorphism of graded rings
$$
R(n,m,Q) \rightarrow R(n,m,Q'), \qquad
T_{ij} \mapsto
T_{\sigma(i), \sigma(j)},
\quad
S_l \mapsto S_l
$$
where the $ij$-th column $w'_{ij}$ of $Q'$
equals
$w_{\sigma^{-1}(i) \sigma^{-1}(j)}$.
Note that in the above we identify
$T_{ij}$ with $T_{ji}$ and $w_{ij}$ with $w_{ji}$.
\end{remark}

\begin{lemma}\label{lem:deleteVariables}
Let $X$ be a smooth full intrinsic Grassmannian 
of Picard number two
defined by its Cox ring $R(n,0,Q)$, where $n \geq 5$ holds, and an ample class $u \in K$. 
Then after suitably permuting variables as in Remark \ref{rem:permutation} the variety 
${X_n := \overline{X}_n^{ss}(u) / \!\! / H}$
is a smooth full intrinsic Grassmannian 
of type $(2, n-1)$,
where
$$
\overline{X}_n := \mathrm{Spec}(R(n-1,0,Q'))
$$
and $Q'$ is the $2 \times \binom{n-1}{2}$ matrix
whose $ij$-th column $w_{ij}'$ equals the $ij$-the column $w_{ij}$ of $Q$
for $1 \leq i < j \leq n-1$.
\end{lemma}
\begin{proof}
The proof runs in two steps: In the first step we show that there exists an $1 \leq l \leq n$ such that there are at least
two weights $w_{ij}$ with $l \notin \left\{i,j\right\}$ in each of $\tau^+$ and $\tau^-$. After a suitable permutation of variables as in Remark \ref{rem:permutation} we may assume $l = n$ and consider the corresponding ring $R(n-1,0, Q')$. 
We then show in the second step that the corresponding variety $X_n$ is indeed smooth.

We may assume that $\tau^+$ contains less weights than $\tau^-$. 
We prove the assertion for $n \geq 6$. Let $w_{i_1j_1}, w_{i_2j_2} \in \tau^+$ be any two weights and choose 
$l \in \left\{1,\ldots, n\right\} \setminus \left\{i_1,j_1, i_2,j_2\right\}$. Note that there are exactly $n-1$ weights $w_{ij}$ with $l \in \left\{i,j\right\}$.
As $\frac{1}{2}\binom{n}{2} - (n-1) \geq 2$
holds for $n \geq 6$ we conclude that there are still at least two weights $w_{ij}$ with $l \notin \left\{i,j\right\}$
left in $\tau^-$.

We come to $n =5$. Here the above choice works as long as $\tau^-$ contains more than $6$ weights. So assume both $\tau^+$ and $\tau^-$ contain exactly $5$ weights. 
In this case there exists at least one $1 \leq l \leq 5$ 
such that there are weights $w_{i_1j_1} \in \tau^+$, $w_{i_2j_2} \in \tau^-$ with $l \in \left\{i_1, j_1\right\}$ and $l \in \left\{i_2, j_2 \right\}$. In particular there are at least two weights $w_{ij}$ with $l \notin \left\{i,j\right\}$ in each of $\tau^+$ and $\tau^-$.

We are left with proving that $X_n$ is smooth.
Due to the structure of the relations any $X_n$-relevant face
$\gamma \in \rlv(X_n)$ is $X$-relevant,
where in the latter case we regard $\gamma$ as a cone in the
ambient space $\QQ^{\binom{n}{2}}$. In particular, 
as $X_n$ is already quasismooth, any $\gamma \in \mathrm{rlv}(X)$ defines a smooth piece in $X_n$.
This completes the proof.
\end{proof}

\begin{remark}\label{rem:type24}
Let $X$ be a smooth full intrinsic Grassmannian of type $(2,4)$. Then $X$ is a smooth full intrinsic quadric, which were classified in \cite{FaHa} for small Picard numbers. Their results show that there is exactly one smooth full intrinsic Grassmannian of type $(2,4)$ with Picard number two: We have 
$\mathcal{R}(X) = R(4,0,Q)$ with 
$$
Q:= [w_{12}, w_{13}, w_{14}, w_{23},w_{24}, w_{34}]
=
\begin{bmatrix}
1 & 1 & 0 & 1 & 0 & 0 
\\
0 & 0 & 1 & 0 & 1 & 1
\end{bmatrix}
$$
and the ample cone equals the positive orthant. 
\end{remark}

\begin{proof}[Proof of Theorem \ref{thm:classification}]
Remark \ref{rem:type24} proves the assertion in the case $n = 4$. Now, assume that any smooth full intrinsic Grassmannian of type~$(2,n-1)$
that has Picard number two
is as described in Theorem \ref{thm:classification}.
We show that this implies that any smooth full intrinsic Grassmannian of type $(2,n)$ with Picard number two is as well as claimed.

For this, let $X$ be 
any smooth full intrinsic Grassmannian of type $(2,n)$ that has Picard number two. 
Then due to Lemma \ref{lem:deleteVariables}
after suitably permuting the variables 
the variety $X_n$ is a smooth intrinsic Grassmannian 
of type $(2,n-1)$
that has Picard number two.
Thus the weights $w_{ij}$ with $j \neq n$
are of the form described in 
Theorem~\ref{thm:classification}.
Moreover, due to homogeneity of the relations 
we have
$$
w_{1n} + w_{23} = w_{2n} + w_{13} + w_{3n} + w_{12}.
$$
Using $w_{12} = w_{13} = w_{23}$
we conclude
$w_{1n} = w_{2n} = w_{3n}$.
Again homogeneity of the relations forces
$$
w_{in} =
w_{1i} + w_{2n} - w_{12}
= w_{1i} + w_{1n} - (1,0).
$$
Thus it is only left to show that
$w_{1n}$ lies either in $\tau^+$ and equals $(1,0)$ or 
lies in $\tau^-$ and is of the
form $(\alpha_n,1)$ for $\alpha_n \in \ZZ$. 
Then suitably renumbering and applying a suitable coordinate change on $\ZZ^2$ the weights and the ample cone are as claimed.

We assume $w_{1n} \in \tau^-$.
Then
$\gamma_{12, 1n} \in \mathrm{rlv}(X)$
leads to 
$$
1 = \mathrm{det}(w_{12},w_{1n}) = y_{1n} = y_{2n} = y_{3n}.
$$
Thus setting $\alpha_n := x_{1n}$
the grading is in the desired form.

Now assume $w_{1n} \in \tau^+$. 
Then
$\gamma_{1k,1n} \in \mathrm{rlv}(X)$
leads to
$
1 = \mathrm{det}(w_{1n},w_{1k}) = 
x_{1n}.
$
We distinguish between the
case $\alpha_{n-1} > 0$ and the 
case $\alpha_{n-1} = 0$.

Assume $\alpha_{n-1} > 0$.
Then $\gamma_{1n-1, 1n} \in \mathrm{rlv}(X)$
leads to
$$
1 =
\mathrm{det}(w_{1n}, w_{1n-1}) = 
1 - y_{1n} \cdot \alpha_{n-1}.
$$
Therefore $y_{1n} = 0$ holds
and the grading is as wanted. 

Now assume $\alpha_{n-1} = 0$.
Due to homogeneity of the relations we have
$$
x_{n-1n}
= 
x_{1n-1} + x_{2n} - x_{12} 
=
0.
$$
This implies $w_{n-1n} \in \tau^-$.
Now, $\gamma_{1n,n-1n} \in \mathrm{rlv}(X)$ leads to
$
1 =
\mathrm{det}(w_{1n},w_{n-1n}) 
= y_{n-1n}
$
and homogeneity of the relations implies
$$
(1,0) 
=
w_{n-1n}+ w_{12} - w_{2n-1}
= w_{1n}.
$$
Again we are in the desired form.

In order to complete the proof it is only left to show that the constellation in Theorem \ref{thm:classification} defines smooth full intrinsic Grassmannians of type $(2,n)$. One directly checks that the weights $w_{ij}$ leave the Plücker relations homogeneous and thus define a grading on $R(n,0)$ as wanted.
In order to show that $X$ is indeed smooth it suffices to show that each minimal face $\gamma_0 \in \mathrm{rlv}(X)$ defines a smooth piece of $X$.
The minimal $X$-relevant faces are precisely the two-dimensional ones. Using Lemma \ref{lem:XbarFaces}~(ii) we obtain that a face $\gamma_{i_1j_1, i_2j_2}\preceq \gamma$
is an $\overline{X}$-face if and only if 
$|\left\{i_1, i_2, j_1, j_2\right\}| = 3$ holds. Moreover,
the $\overline{X}$-face 
$\gamma_{i_1j_1, i_2j_2}$ is $X$-relevant if and only if
after suitably renumbering
$w_{i_1j_1} \in \tau^+$ and $w_{i_2j_2} \in \tau^-$ holds.
Thus the constellation of weights given in Theorem~\ref{thm:classification} forces 
$i_1< j_1 < k \leq j_2$ and we have either 
$i_2 = j_1$ or $i_2 = i_1.$
With
$\mathrm{det}((1,0), (\alpha_{j_2}, 1)) = 1$
the assertion follows. 
\end{proof}

\begin{proof}[Proof of Theorem \ref{thm:classNonFull}]
Let $X$ be any smooth intrinsic Grassmannian of type $(2,n)$ with Picard number two.
Due to Remark \ref{rem:conesAndCo} and Lemma~\ref{lem:Pic2Weights} we have one of the following constellations of weights in the effective cone:

\vspace{2mm}
\noindent
\emph{Case 1: 
We have $m \geq 0$ and there are at least two weights of the form $w_{ij}$ in each of $\tau^+$ and $\tau^-$.}

\noindent
In this situation the weights $w_{ij}$ are of the form as in Theorem \ref{thm:classification}.
If $m = 0$ holds, we are in Type~1. So assume $m> 0$.
We distinguish between the following two cases:

\vspace{1mm}
\noindent
\emph{Case 1.1: We have $w_l \in \tau^+$ for all $1 \leq l \leq m$.}

\noindent
Applying Remark \ref{rem:conesAndCo2} to $\gamma_{l, 1k} \in \rlv(X)$ leads to 
$x_l = 1$. Assume $y_l = 0$ holds for at least one $1 \leq l \leq m$. 
We claim that this forces $k = n$. Assume not.
Then we have
$$
w_{n-1,n} 
=
(\alpha_{n-1} + \alpha_n -1, 2).
$$
But $\mathrm{det}(w_l,w_{n-1,n}) = 2$
contradicts smoothness of $X$ as
$\gamma_{l, n-1n} \in \rlv(X)$.
Thus we have $k = n$ and suitably renumbering the free variables and applying a suitable coordinate change on $\ZZ^2$ we end up in Type~2.

So assume $y_l \neq 0$ for all $1 \leq l \leq m$. 
In this case $\gamma_{l, 1j} \in \rlv(X)$ for $j \geq k$ forces
$\alpha_k = \ldots = \alpha_n = 0$.
If $k = n$ holds, we are again in Type~2.
So assume $k < n$. 
Then
$\gamma_{l, n-1n} \in \mathrm{rlv}(X)$ yields
$$
1 = \mathrm{det}(w_{l}, w_{n-1,n}) = 2 + y_l.
$$
and thus $y_l = -1$ for all $1 \leq l \leq m$. 
Applying the coordinate change on $\ZZ^2$ that sends $(1,-1)$ to $(1,0)$ and
$(-1,2)$ to $(0,1)$ we end up in Type~3.

\vspace{1mm}
\noindent
\emph{Case 1.2: We have $w_l \in \tau^-$ for all $1 \leq l \leq m$}

\noindent
In this case, applying Remark \ref{rem:conesAndCo2} to $\gamma_{l, 12} \in \rlv(X)$ leads to 
$y_l = 1$. We set $\beta_l:=x_l$ 
for $1 \leq l \leq m$. 
Applying a suitable unimodular coordinate change we may assume $\beta_l \geq 0$ holds for all
$1 \leq l \leq m$ and $(0,1)$ occurs among the weights
$w_{ij}$ and $w_l$. 
Now suitably renumbering the free variables we obtain
$0 \leq \beta_1 \leq \ldots \leq \beta_m$ and we are in the situation of Type~1.

\vspace{2mm}
\noindent
\emph{Case 2: 
We have $m \geq 1$, 
with $w_l \in \tau^+$ for $1 \leq l \leq m$ and there is exactly one more weight $w_{ij}$ in $\tau^+$.}

\noindent
After suitably permuting variables as in Remark \ref{rem:permutation} we may assume $w_{12}, w_l \in \tau^+$, where $1 \leq l \leq m$. 
Now, applying Remark \ref{rem:conesAndCo2} to $\gamma_{1,13} \in \mathrm{rlv}(X)$ and afterwards to $\gamma_{1,ij}\in \mathrm{rlv}(X)$,
where $(1,2) \neq (i,j) \neq (1,3)$ 
we obtain
$$
w_{1} =(1,0), \ w_{13} = (0,1) \ \text{ and } y_{ij} = 1 \text{ whenever } j \geq 3
$$
Then $w_{12} + w_{34} = w_{13} + w_{24}$
and $\gamma_{12,13} \in \rlv(X)$ 
yields $w_{12} = (1,1)$.
Using 
$\gamma_{12,ij} \in \rlv(X)$ for 
$i \leq 2 <j$ 
we conclude
$w_{ij} = (0,1)$ for $i \leq 2 <j$.
Using homogeneity of the relations we obtain
$$
w_{ij} = w_{1i} + w_{2j} - w_{12} = (-1,1)
\text{ whenever } i \geq 3.
$$
Applying a suitable coordinate change on $\ZZ^2$ we end up in Type~4.

\vspace{2mm}
\noindent
\emph{Case 3: We have $m \geq 2$ with $w_l \in \tau^+$ for $1 \leq l \leq m$ and $w_{ij} \in \tau^-$ for all $(i,j)$.}

\noindent
Applying Remark \ref{rem:conesAndCo2}
to $\gamma_{1,12} \in \mathrm{rlv}(X)$ and 
afterwards to $\gamma_{1,ij} \in \rlv(X)$ and 
$\gamma_{l,12} \in \rlv(X)$, where $(i,j) \neq (1,2)$ and
$1 \leq l \leq m$,
we obtain
$$
w_{12} = (0,1), \ w_{1} = (1,0),
\ 
x_l = 1 \text{ and } y_{ij} = 1.
$$
We distinguish between the following cases:

\vspace{1mm}
\noindent
\emph{Case 3.1: We have $y_l = 0$ for all $1 \leq l \leq m$.}

\noindent
In this case set $\alpha_j := x_{3j}$ for $j \geq 4$. Then 
homogeneity of the relations forces
$$
(\alpha_j, 2) = w_{12} +w_{3j} = w_{13} + w_{2j} = w_{23} + w_{1j}
$$
and $w_{ij} = w_{1i} + w_{2j} - (0,1)$.
Applying a suitable coordinate change on $\ZZ^2$ we may assume $x_{ij} \geq 0$ for all $1 \leq i < j \leq n$ and at least 
one $x_{ij}$ equals $0$. 
Applying a suitable permutation of variables
as in Remark \ref{rem:permutation} we may thus assume 
$0 = x_{12} \leq x_{13} \leq x_{23}$ and we are in Type~5.

\vspace{1mm}
\noindent
\emph{Case 3.2: There exists $1 \leq l \leq m$ with $y_l \neq 0$.}

\noindent
In this case $\gamma_{l, ij} \in \rlv(X)$
forces $w_{ij} = (0,1)$ for all $1 \leq i < j \leq n$.
Suitably renumbering the free variables we may furthermore assume 
$y_1 \leq y_2 \leq \ldots \leq y_m$.
Moreover, applying a suitable coordinate change on $\ZZ^2$ we
obtain $0 = y_1 \leq \ldots \leq y_m$ and 
we are in Type~6.

We complete the proof by showing that indeed all 
constellations in Theorem \ref{thm:classNonFull} define smooth intrinsic Grassmannians of type $(2,n)$. One directly checks that the weights $w_{ij}$ leave the Plücker relations homogeneous and thus define a grading on $R(n,m)$ as wanted.
We show that the minimal $X$-relevant faces, that means the two-dimensional ones, define a smooth piece of $X$. We go through the six types:

\vspace{1mm}
\noindent
\emph{Type~1:}
Here the two-dimensional $X$-relevant faces are either of the form
$\gamma_{i_1j_1,l}$ with $j_1 < k$ or of the form
$\gamma_{i_1j_1,i_2j_2}$ 
with
$j_1 < k \leq \ j_2$
and
$i_2 \in \left\{i_1, j_1\right\}$.
Using Lemma \ref{lem:XbarFaces}~(iv), we conclude that these faces define smooth pieces of $X$ as $w_{i_1j_1} = (1,0)$, $w_l = (\beta_l,1)$ and $w_{i_2,j_2} = ( \alpha_{j_2},1)$ holds.

\vspace{1mm}
\noindent
\emph{Type~2:}
Here the two-dimensional $X$-relevant faces are either of the form
$\gamma_{i_1n,l}$ with $i < n$ or of the form
$\gamma_{i_1n,i_2j_2}$ with $i_1,i_2,j_2 < n$
and $i_2 = i_1$ or $j_2 = i_1$.
These faces define smooth pieces of $X$ as $w_{i_1n} = (1,0)$, $w_l = (\beta_l,1)$ and $w_{i_2j_2} = ( \alpha,1)$ holds. 

\vspace{1mm}
\noindent
\emph{Type~3:}
Here the two-dimensional $X$-relevant faces are either of the form
$\gamma_{l,ij}$ with $k \leq j$ or of the form
$\gamma_{i_1j_1,i_2j_2}$ with 
$j_1 < k \leq j_2$ and
$i_2 = i_1$ or $i_2 = j_1.$
The former face defines a smooth piece of $X$
as $w_l = (1,0)$ holds and $w_{ij}$ equals either $(1,1)$ or $(0,1)$.
The latter face defines a smooth piece of $X$ 
as $w_{i_1j_1} = (2,1)$ and 
$w_{i_2j_2} = (1,1)$ holds. 

\vspace{1mm}
\noindent
\emph{Type~4:}
Here the two-dimensional $X$-relevant faces are either of the form
$\gamma_{l, ij}$ with $(i,j) \neq (1,2)$ or of the form
$
\gamma_{12,ij}$
with $i \in \left\{1,2\right\}$.
The former face defines smooth pieces of $X$ as $w_{l} = (1,0)$ and $w_{ij}$ equals either $(1,1)$ or $(0,1)$. The latter face defines a smooth piece of $X$ as $w_{12} = (2,1)$ holds and $w_{ij}$ with
$i \in \left\{1,2\right\}$ equals $(1,1)$.

\vspace{1mm}
\noindent
\emph{Type~5:}
Here the two-dimensional $X$-relevant faces are of the form $\gamma_{l, ij}$.
As $w_l= (1,0)$ and $y_{ij} = 1$ holds for all $1 \leq i < j \leq n$ these faces define smooth pieces of $X$. 

\vspace{1mm}
\noindent
\emph{Type~6:}
Here the two-dimensional $X$-relevant faces are of the form $\gamma_{ij,l}$.
As $w_l= (\beta_l,1)$ and $w_{ij} = (1,0)$ holds for all $1 \leq i < j \leq n$ these faces define smooth pieces of $X$. 
\end{proof}

\begin{proof}[Proof of Corollaries \ref{cor:FanoClassification} and \ref{cor:FanoNonFull}]
Let $X$ be a smooth intrinsic Grassmannian of type $(2,n)$
that has Picard number two. 
Using Theorem~\ref{thm:classNonFull}
and 
Proposition~\ref{prop:AnticanClass}
we obtain the anticanonical classes in each of the types as

{
\renewcommand*{\arraystretch}{2.5}
\begin{longtable}{c||c}
Type &
$- \mathcal{K}_X$
\\
\hline
\hline
1
&
$
 \left(
 \sum\limits_{j=l}\limits^n 2 \alpha_j + \sum\limits_{l=1}\limits^m \beta_l + 2k - n - 2, \ 2(n-k+1)
 \right)
 $
\\
%%%%%%%%%%%%%%%%%%%%%%%%%%%%%%%%%%%%%%%%%%%%%%%%%%
%%%%%%%%%%%%%%%%%%%%%%%%%%%%%%%%%%%%%%%%%%%%%%%%%%
%%%%%%%%%%%%%%%%%%%%%%%%%%%%%%%%%%%%%%%%%%%%%%%%%%
\hline
2
&
$
 \left(
 \sum\limits_{l=1}\limits^m \beta_l + (n-2)\alpha + 2, \ 
 n+m-2
 \right)
 $
\\
%%%%%%%%%%%%%%%%%%%%%%%%%%%%%%%%%%%%%%%%%%%%%%%%%%
%%%%%%%%%%%%%%%%%%%%%%%%%%%%%%%%%%%%%%%%%%%%%%%%%%
%%%%%%%%%%%%%%%%%%%%%%%%%%%%%%%%%%%%%%%%%%%%%%%%%%
\hline
3
&
$
 \left(
 2 (k-1) + m, \ n
 \right)
 $
\\
%%%%%%%%%%%%%%%%%%%%%%%%%%%%%%%%%%%%%%%%%%%%%%%%%%
%%%%%%%%%%%%%%%%%%%%%%%%%%%%%%%%%%%%%%%%%%%%%%%%%%
%%%%%%%%%%%%%%%%%%%%%%%%%%%%%%%%%%%%%%%%%%%%%%%%%%
\hline
4
&
$
 \left(
 4 + m, \ n
 \right)
 $
\\
%%%%%%%%%%%%%%%%%%%%%%%%%%%%%%%%%%%%%%%%%%%%%%%%%%
%%%%%%%%%%%%%%%%%%%%%%%%%%%%%%%%%%%%%%%%%%%%%%%%%%
%%%%%%%%%%%%%%%%%%%%%%%%%%%%%%%%%%%%%%%%%%%%%%%%%%
\hline
5
&
$
 \left(
 \sum\limits_{j\geq 4} 2 \alpha_j - (n-4)(b_1 + b_2) +m, \ 
 n
 \right)
 $
\\
%%%%%%%%%%%%%%%%%%%%%%%%%%%%%%%%%%%%%%%%%%%%%%%%%%
%%%%%%%%%%%%%%%%%%%%%%%%%%%%%%%%%%%%%%%%%%%%%%%%%%
%%%%%%%%%%%%%%%%%%%%%%%%%%%%%%%%%%%%%%%%%%%%%%%%%%
\hline
6
&
$
 \left(
n + \sum\limits_{l=1}\limits^m \beta_l, \ m
 \right)
 $
\\
\end{longtable}}
\noindent 
and the semiample cones
are described in Theorem \ref{thm:classNonFull}. 
As $X$ is Fano if and only if $- \mathcal{K}_X$ lies in the relative interior of the semiample cone 
$\mathrm{SAmple}(X)$, and 
$X$ is truly almost Fano if and only if $- \mathcal{K}_X$ lies on the boundary of $\mathrm{SAmple}(X)$ we obtain the statement via direct calculation.
The full intrinsic Grassmannians described in Theorem \ref{thm:classification} are of Type~1 with $m =0$. In particular the criteria of Corollary \ref{cor:FanoClassification} directly follow from the criteria for Type~1 by setting $m=0$. 
\end{proof}

\begin{proof}[Proof of Corollary \ref{cor:NumberOfFanos}]
In order to get the number of smooth Fano full intrinsic Grassmannians of type $(2,n)$ with Picard number two, 
we calculate all possible values for $k$ and $\alpha_j$,
where $k \leq j \leq n$, fulfilling the inequality of Corollary~\ref{cor:FanoClassification}:
\begin{equation}\label{equ1}
    (n-k)\alpha_n < \sum_{i=k+1}^{n-1} \alpha_i - \frac{1}{2} n + k - 1.
\end{equation}
As $\alpha_j \leq \alpha_n$ holds for all $j \leq n$, we conclude $0 \leq \alpha_n < k - \frac{n}{2} -1$
and thus $1 + n/2 < k$.
Moreover, we obtain
$$
(\ref{equ1}) \Longleftrightarrow
\sum_{i = {k+1}}^{n-1} (\alpha_n - \alpha_j) \leq k - \lfloor \frac{n}{2} \rfloor - 2 - \alpha_n.
$$
Therefore for fixed $k$ and fixed $\alpha_n$ we need to count all sequences
$0 \leq \alpha_{k+1} \leq \ldots \leq \alpha_{n-1} \leq \alpha_n$ such that the inequality on the right hand side holds. This proves the formula.

In order to complete the proof we show that for fixed $n$, different values of $k$ and $\alpha_j$ 
give rise to non-isomorphic varieties.
For this
let $4 \leq k \leq n$ and $4 \leq k' \leq n$ be integers, fix sequences
$$
0 =  \alpha_k \leq \alpha_{k+1} \leq \ldots \leq \alpha_n
\quad
\text{and}
\quad
0 =  \alpha'_{k'} \leq \alpha'_{k'+1} \leq \ldots \leq \alpha'_n.
$$
and denote with $X$ resp. $X'$ 
the corresponding varieties as in Theorem \ref{thm:classification}
with their respective Cox rings $R$ and $R'$. 

Assume that $X$ and $X'$ are isomorphic.
Then there exists a graded isomorphism $(\varphi,\psi)\colon (R,\mathbb{Z}^2)\rightarrow (R',\mathbb{Z}^2)$, where $\psi$ maps the effective, the moving and the semiample cone of $X$ onto that of $X'$.
We show that the existence of such an isomorphism yields equal sets of defining data.

Assume $k=n$ and thus $\alpha_n = 0$ holds.
We make use of the set of generator degrees
$$\Omega_R := \left\{w\in K; \ R_w \not\subseteq R_{<w} \right\}.$$
Note that any graded isomorphism from $R$ to $R'$ maps $\Omega_R$ onto $\Omega_{R'}$.
As $k = n$ holds, we have  $|\Omega_R| = 2$. 
Thus, as $k' < n$ would imply $|\Omega_{R'}| > 2$, we obtain $k' = n$
and thus $\alpha'_{k'} = \alpha'_n = 0$.

Now, assume $k < n$ holds.
As $\psi$ maps the boundaries of the effective and of the semiample cone of $X$ onto that of $X'$
we conclude that $\psi$ is the 
identity.
In particular, we obtain $\dim(R_{(1,0)})=\dim(R'_{(1,0)})$ and thus $k=k'$. 

Now, assume the sequences of integers $ 0 =  \alpha_k \leq \alpha_{k+1} \leq \ldots \leq \alpha_n$
resp. $ 0 =  \alpha'_{k} \leq \alpha'_{k+1} \leq \ldots \leq \alpha'_n$
are different. Then there exists a minimal $j < k$ with
$\alpha_j \neq \alpha'_j$ and we may assume $\alpha_j < \alpha'_j$.
We conclude
$\alpha_j = \alpha_{j-1}$, because else $\Omega_R$ is not send onto $\Omega_{R'}$.
But this implies that 
the homogeneous component
$R_{(\alpha_j,1)}$ contains all generators that are contained in
$R'_{(\alpha_j,1)} = R'_{(\alpha'_{j-1},1)}$ and new generators, namely
$T_{ij}$, where $i < k$ holds. As these 
are not in the span of $R'_{(\alpha'_{j-1},1)}$
we end up with a contradiction:
$$
\mathrm{dim}(R_{(\alpha_j,1)}) > \mathrm{dim}(R'_{(\alpha_j,1)}).
$$
\end{proof}

\section{Geometry and Fujita's freeness conjecture}\label{sec:geometry}
In this section, we start by discussing geometric aspects of the smooth intrinsic Grassmannians of type $(2,n)$ of Picard number 2, classified in Theorem~\ref{thm:classNonFull}.
Then in Corollary \ref{cor:Fujita}, we  verify Fujita’s freeness conjecture for these varieties. We briefly recall the relevant backgrounds from birational geometry, see \cite{Casa}.
Let $X$ be any normal projective variety and let $D$ be a Weil divisor on $X$. If $[D] \in \mathrm{SAmple}(X)$  then there exist a morphism, known as \emph{contraction}:
$$
\varphi_D \colon X \rightarrow X(D) := \mathrm{Proj}\left(\bigoplus_{n \in \ZZ_{\geq 0}}\Gamma(X, \mathcal{O}_X(nD)\right).
$$
The contraction is called \emph{elementary} if the Picard number of  $X(D)$ is one less than the Picard number of  $X$. There are three possible types of elementary contractions:
\begin{enumerate}
    \item 
    If the class $[D]$ lies on the boundary of the effective cone, then the dimension of $X(D)$ is strictly less than that of $X$ and we call $\varphi_D$ \emph{of fiber~type}.
    \item
    If the class $[D]$ lies on the boundary of the moving cone but not on the boundary of the effective cone, then $\varphi_D$ is birational and contracts precisely one divisor.
    In this situation $\varphi_D$ is called a \emph{birational divisorial contraction.}
    \item
    If the class $[D]$ lies in the interior of the moving cone, then $\varphi_D$ is birational and contracts a subvariety of codimension at least two.
    In this situation $\varphi_D$ is called a \emph{birational small contraction.}
\end{enumerate}

Now, let $X$ be a smooth intrinsic Grassmannian of type $(2,n)$ with 
Picard number two.
Then Construction~\ref{constr:intrinsGrassmannian}
provides an embedding $X \subseteq Z$ into a toric variety $Z$. 
In each of the six types 
the semiample cone of $X$ and $Z$ coincide. Let $[E] \in \mathrm{SAmple}(Z)= \mathrm{SAmple}(X)$ be a  divisor class, where $E$ is a toric divisor on $Z$, and let $D$ denote its restriction to $X$. Then we obtain a commutative diagram
$$
\xymatrix{
X
\ar@{}[r]|{\subseteq}
\ar[d]_{\varphi_D}
&
{Z}
\ar[d]^{\varphi_E}
\\
X(D)
\ar@{}[r]|{\subseteq}
&
Z(E).
}
$$

\begin{remark}\label{rem:splitVectorBundle}
We consider the varieties $X$ of Types~1, 2, 5 and 6 and work  with the toric embedding $X \subseteq Z$ provided by Construction~\ref{constr:intrinsGrassmannian}. 
We denote the number of weights in $\tau^+$ with $s$ and the number of weights in $\tau^-$ with $t$. All weights in $\tau^+$
are of the form $(1,0)$. In $\tau^-$,
we denote with $t_1$ the number of weights of the form $(a_i,1)$, where we may assume
$a_1 \leq \ldots \leq a_{t_1}$. 
Moreover, we have $t_2 := t- t_1$ weight vectors of the form $(a_{t_1 + i},2)$ with $a_{t_1+1} \leq \ldots \leq a_{t_1 + t_2}$:

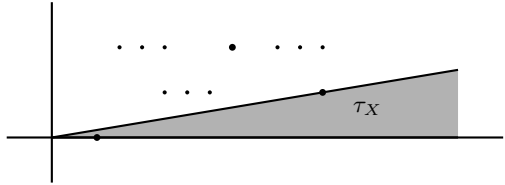
\begin{center}
\begin{tikzpicture}[scale=0.6]
%semiample cone
\path[fill=gray!60!] (0,0)--(9,1.5)--(9,0);
\draw (7,1) circle (0ex) node[below]{\small $\tau_X$};
%rays of the semiample cone \tau_X
\draw[-,thick] (0,0)--(9,1.5);
\draw[-,thick] (0,0)--(9,0);
%Cl(X)-degrees of type (i)
\path[fill, color=black] (1,0) circle (0.5ex);
%Cl(X)-degrees of type (ii)
\path[fill, color=black] (6,1) circle (0.5ex);
\path[fill, color=black] (2.5,1) circle (0.3ex);
\path[fill, color=black] (3,1) circle (0.3ex);
\path[fill, color=black] (3.5,1) circle (0.3ex);
%Cl(X)-degrees of type (iii) for i,j \geq k
\path[fill, color=black] (4,2) circle (0.5ex);
\path[fill, color=black] (1.5,2) circle (0.3ex);
\path[fill, color=black] (2,2) circle (0.3ex);
\path[fill, color=black] (2.5,2) circle (0.3ex);
\path[fill, color=black] (5,2) circle (0.3ex);
\path[fill, color=black] (5.5,2) circle (0.3ex);
\path[fill, color=black] (6,2) circle (0.3ex);
%axes
\draw[-,thick] (-1,0)--(10,0);
\draw[-,thick] (0,-1)--(0,3);
\end{tikzpicture}   
\end{center}
Take any toric divisor $E$ on $Z$ representing the divisor class $(1,0)$.
Then 
$\varphi_E \colon Z \rightarrow Z(E) = \PP^{s-1}$
is 
of fiber type. More precisely,
$Z$
a weighted projectivized split vector bundle
of rank $t-1$ over $\PP^{s-1}$ whose fibers are isomorphic to weighted projective spaces $\PP_{1^{t_1}, 2^{t_2}}$; see e.g.\ \cite{JPMullett} and \cite{CoxLitSch2011}*{Section 7.3} for the case $t_2 =0$.
More precisely, we have
$$
Z \cong \mathrm{Proj}
(\mathcal{O}_{\PP^{s-1}}(a_1) 
\oplus
\mathcal{O}_{\PP^{s-1}}(a_2)
\oplus
\ldots
\oplus
\mathcal{O}_{\PP^{s-1}}(a_t)).
$$
\end{remark}

For the subsequent description of the geometry of the varieties $X$ of Types 1, 2, 5 and 6 we work in the above notation and furthermore denote the restriction of the toric divisor $E$ onto $X$ by $D$. 

\vspace{3mm}
\noindent
\textbf{Type~1.}
\emph{Here, $X$ admits a locally trivial fibration with fibers of dimension ${2(n-k) +3}$ in  $\PP_{1^{t_1},2^{t_2}}$, where 
$t_1 := (k-1)(n-k+1) +m$ and
$t_2 := \binom{n-k+1}{2}$.}

\vspace{-1mm}
\begin{proof}
In this case we have $s = \binom{k-1}{2}$ and $t = t_1 + t_2$. We obtain 
$$
X(D) = \mathrm{Gr}(2, k-1) \subseteq  \PP^{\binom{k-1}{2} -1} = Z(E).$$
As described in Remark \ref{rem:splitVectorBundle}, the fibers of $\varphi_E \colon Z \rightarrow Z(E)$ 
are weighted projective spaces $\PP_{1^{t_1},2^{t_2}}$.
One directly calculates the dimension of the fibers of $X$.
\end{proof}

\vspace{3mm}
\noindent
\textbf{Type~2.}
\emph{Here, $X$ admits a locally trivial fibration with fibers isomorphic to $\PP^{n+m-3}$.}

\vspace{-1mm}
\begin{proof}
In this case we have $s = n-1$ and $t = t_1 = \binom{n-1}{2} +m$. We obtain $X(D) = Z(E) = \PP^{n-2}$. On each fiber
$\varphi^{-1}_E(z)$ the relations $g_{I}$ with $n \in I$
become linear and one directly checks that the ideal generated by these linear relations contains the relations $g_I$ where $n \notin I$.
In particular the fibers are linear subspaces of codimension $\binom{n-2}{2}$ in $\PP^{\binom{n-1}{2}+ m-1}$.
\end{proof}

\vspace{3mm}
\noindent
\textbf{Type~5.}
\emph{Here, $X$ admits a locally trivial fibration with fibers $\mathrm{Gr}(2,n) \subseteq \PP^{ \binom{n}{2} -1}.$}

\vspace{-1mm}
\begin{proof}
In this case we have $s = m$
and $t = t_1 =\binom{n}{2}$. We obtain
$X(D) = Z(E) = \PP^{m-1}$. As the local trivialization of the bundle projection $Z \rightarrow Z(E)$ given in \ref{rem:splitVectorBundle} respects the relations $g_I \in I_{2,n}$ 
the variety $X$ is a locally trivial fibration with fibers $\mathrm{Gr}(2,n) \subseteq \PP^{ \binom{n}{2} -1}.$
\end{proof}

\vspace{3mm}
\noindent
\textbf{Type~6.}
\emph{Here, $X$ is the projectivized split vector bundle
$$
\PP(\mathcal{O}_{\mathrm{Gr}(2,n)}(\beta_1)
\oplus
\mathcal{O}_{\mathrm{Gr}(2,n)}(\beta_2)
\oplus
\ldots
\oplus
\mathcal{O}_{\mathrm{Gr}(2,n)}(\beta_m)).
$$}

\vspace{-5mm}
\begin{proof}
In this case we have $s = \binom{n}{2}$
and $t = t_1 = m$. We obtain $X(D) = \mathrm{V}(I_{2,n}) \subseteq  \PP^{\binom{n}{2}-1} = Z(E)$. Moreover we have
$\varphi^{-1}(\mathrm{V}(I_{2,n})) = X \subseteq Z$. 
Using Remark \ref{rem:splitVectorBundle}, we obtain that
$X$ is as claimed. 
\end{proof}

We turn to the Types~3 and 4, where we obtain divisorial contractions.
We note that for $m=0$ Type~3 is a subcase of Type~1. 

\vspace{3mm}
\noindent
\textbf{Type~3.}
\emph{If $m=1$ holds, $X$ admits a birational divisorial contraction with center of contraction $\mathrm{V}(I_{2,k-1}, T_{ij}; \  j \geq k) \in \PP^{\binom{n}{2} -1}$}

\vspace{-1mm}
\begin{proof}
Let $E$ on $Z$ resp. $D$ on $X$ be the divisors corresponding to 
the generator $T_{12}$. Then 
$\varphi_E$ resp. $\varphi_D$ contract the respective divisors corresponding to $S_1$. We have $X(D) = \mathrm{V}(I_{2,n}) \subseteq \PP^{\binom{n}{2} -1} = Z(E)$.
Moreover, the fan of $Z$ arises out of the fan of $Z(E)$ by a non-barycentric subdivision of the cone over the rays corresponding to $w_{ij}$, where $j \geq k$ holds.
Thus the center of the modification $X \rightarrow X(E)$ is the Grassmannian  $\mathrm{Gr}(2,k-1) \cong \mathrm{V}(I_{2, k-1}, T_{ij}; \ j \geq k)  \in \PP^{\binom{n}{2} -1}$
\end{proof}

\vspace{3mm}
\noindent
\textbf{Type~4.}
\emph{If $m=1$ holds, $X$ admits a birational divisorial contraction with center of contraction $\mathrm{V}(T_{ij};\ j \geq 3) \in \PP^{\binom{n}{2} -1}$}

\vspace{-1mm}
\begin{proof}
Let $E$ on $Z$ resp. $D$ on $X$ be the divisors corresponding to 
the generator $T_{12}$. Then 
$\varphi_E$ resp. $\varphi_D$ contract the respective divisors corresponding to $S_1$. We have $X(D) = \mathrm{V}(I_{2,n}) \subseteq \PP^{\binom{n}{2} -1} = Z(E)$.
Moreover, the fan of $Z$ arises out of the fan of $Z(E)$ by a non-barycentric subdivision of the cone over the rays corresponding to $w_{ij}$, where $j \geq 3$.
Thus the center of the modification $X \rightarrow X(E)$ is the point $\mathrm{V}(T_{ij}; \ j \geq 3)  \in \PP^{\binom{n}{2} -1}$.
\end{proof}

\begin{remark}
Let $X$ be of Type~4 and assume $n= 4$ holds. Then we are in Type~3 of the classification
of smooth intrinsic quadrics of Picard number two 
in \cite{FaHa}*{Thm. 1.1}. In particular in this case
$X$ is the blowing-up of the projective space $\PP^{3 + m}$ centered at ${V(T_{13}T_{24} - T_{14}T_{23}, S_1, \ldots, S_m) \subseteq \PP^{3+m}.}$
\end{remark}

\begin{corollary}\label{cor:Fujita}
Every smooth intrinsic Grassmannian $X$ of type $(2,n)$ with Picard number two fulfills Fujita's freeness conjecture:
For every $m \geq \mathrm{dim}(X) +1$
and every ample divisor $D$ on $X$ 
the divisor $K_X + m \cdot D$ is base point free,
where $K_X$ is a canonical divisor of $X$.
\end{corollary}
\begin{proof}
In \cite{FuFreeness} Fujita proved that under the above assumptions any of the above divisors $K_X +m\cdot D$ 
is numerical effective. Thus it suffices to prove that any numerical effective divisor of a smooth intrinsic Grassmannian of type $(2,n)$ is base point free.

We consider the monoid $\mathrm{BPF}(X) \subseteq \mathrm{Cl}(X) = K$ of divisor classes admitting a base point free representative. Then due to \cite{ArDeHaLa}*{Prop 3.3.2.8} we have
$$
\mathrm{BPF}(X) = \bigcap_{\gamma_0 \in \mathrm{rlv}(X)} Q(\gamma_0 \cap \ZZ^{\binom{n}{2} + m}).
$$
In particular, the cone over $\mathrm{BPF}(X)$ 
equals the semiample cone $\mathrm{SAmple}(X)$,
which in turn equals the cone of numerically effective divisors as $X$ is a Mori dream space. 
Thus it is only left to show that the monoid 
$\mathrm{BPF}(X)$ is saturated. 
We show that all minimal faces $\gamma_0 \in \mathrm{rlv}(X)$ are saturated. This proves the claim 
as intersections of saturated monoids are saturated.

Due to Lemma \ref{lem:Pic2Weights} all minimal $X$-relevant faces are two-dimensional. 
Let $\gamma_0 = \mathrm{cone}(e, e')$ be such a face.
As $X$ is smooth, Lemma \ref{lem:XbarFaces}~(iv) yields that 
$Q(e)$ and $Q(e')$ form a $\ZZ$-basis of $K = \mathrm{Cl}(X)$. In particular $Q(\gamma_0 \cap \ZZ^{\binom{n}{2} + m})$ is saturated as desired. 
\end{proof}

\bibliographystyle{amsalpha}
\bibliography{References}
\end{document}